\documentclass[11pt]{amsart}
\usepackage{amsmath}
\usepackage{bbding}
\usepackage{tikz}
\usepackage{amsmath,amssymb}
\theoremstyle{plain}
\newtheorem{theorem}{Theorem}[section]
\newtheorem{lemma}[theorem]{Lemma}

\newtheorem{corollary}[theorem]{Corollary}

\newtheorem{assumption}[theorem]{Assumption}

\theoremstyle{definition}

\newtheorem{example}[theorem]{Example}
\numberwithin{equation}{section}
\def\be{\begin{equation}}
\def\ee{\end{equation}}

\begin{document}

\title[Blow-up Sets of Ricci Curvatures]
{Blow-up Sets of Ricci Curvatures of\\ Complete Conformal Metrics}

\author[Han]{Qing Han}
\address{Department of Mathematics\\
University of Notre Dame\\
Notre Dame, IN 46556, USA} \email{qhan@nd.edu}
\author[Shen]{Weiming Shen}
\address{School of Mathematical Sciences\\
Capital Normal University\\
Beijing, 100048, China}
\email{wmshen@aliyun.com}
\author[Wang]{Yue Wang}
\address{School of Mathematical Sciences\\
Capital Normal University\\
Beijing, 100048, China}
\email{yuewang37@aliyun.com}

\begin{abstract}
A version of the singular Yamabe problem in smooth domains in a closed manifold
yields complete conformal metrics with negative constant scalar curvatures.
In this paper, we study the blow-up phenomena of Ricci curvatures of these metrics
on domains whose boundary is close to a certain limit set of a lower dimension.
 We will characterize the blow-up set according to the Yamabe invariant of the underlying manifold.
In particular, we will prove that all points in the lower dimension part of the limit set belong to the blow-up set
on manifolds not conformally equivalent to the standard sphere
and that all but one point in the lower dimension part of the limit set belong to the blow-up set
on manifolds conformally equivalent to the standard sphere.
In certain cases, the blow-up set can be the entire manifold.
We will demonstrate by examples that these results are optimal.
\end{abstract}

\thanks{The first author acknowledges the support of NSF Grant DMS-2305038. The second author acknowledges the support of NSFC Grant 12371208 and NSFC Grant 11901405. The third author acknowledges the support of NSFC Grant 12371236 and NSFC Grant 12001383.
%The second author acknowledges the support of NSFC Grant 11901405.
%The third author acknowledges the support of the Beijing Advanced Innovation Center for Imaging Theory
%and Technology and the key
%research project of the Academy for Multidisciplinary Studies, Capital Normal University.
}
%\date{\today}
\maketitle

\section{Introduction}\label{sec-Intro}

A version of the singular Yamabe problem in bounded domains in the compact manifold can be formulated as follows.
Given a compact Riemannian manifold $(M,g)$ of dimension $n\ge 3$ without boundary
and a smooth submanifold $\Gamma$ in $M$,
does there exist a complete conformal metric on $M\setminus \Gamma$
with a  {\it negative} constant scalar curvature $-n(n-1)$?
For $(M,g)=(S^n, g_{S^n})$,
Loewner and Nirenberg \cite{Loewner&Nirenberg1974} proved that there exists a
complete conformal metric on $S^n\setminus \Gamma$  with the
constant scalar curvature $-n(n-1)$
if and only if dim$(\Gamma)>(n-2)/2$.
Aviles and McOwen \cite{AM1988DUKE} proved a similar result for the general compact manifold $(M,g)$.
As a consequence, we can take the dimension of the submanifold to be $n-1$
and conclude the following result. In any compact Riemannian manifold with a smooth boundary,
there exists a complete conformal metric with a negative constant scalar curvature $-n(n-1)$.
 Such a metric is referred to as a {\it singular Yamabe metric} in this paper.

 The singular Yamabe metrics in smooth domains can be viewed as generalizations
of the Poincar\'e metric in the unit ball in Euclidean space,
and have sectional curvatures asymptotically equal to $-1$ near the boundary.
It is natural to investigate whether they have negative sectional curvatures or negative Ricci curvatures in the entire domain,
or whether their Ricci curvatures can be controlled in a uniform way.
Conditions on a bound or a lower bound of Ricci curvatures
are widely used in geometry and analysis on manifolds.
For example, they play an important role in Yau's
gradient estimate for harmonic function \cite{SchoenYau1994}, Li-Yau's heat kernel estimates \cite{SchoenYau1994},
Bishop-Gromov's volume growth estimates \cite{Petersen2006},
Gromov's precompactness theorem for a family of manifolds \cite{Petersen2006}, Cheeger-Colding theory\cite{CheegerColding1997},
and Cheeger-Colding-Tian \cite{CheegerColdingTian2002}, to name a few.

There are many works concerning the singular Yamabe metrics.
Andersson, Chru\'sciel, and Friedrich \cite{ACF1982CMP} and Mazzeo \cite{Mazzeo1991}
established polyhomogeneous expansions for conformal factors of the singular Yamabe metrics.
Graham \cite{Graham2017} studied volume renormalizations for singular
Yamabe metrics and characterized the coefficient of the first logarithmic terms in the expansions
for conformal factors of the singular Yamabe metrics by a variation of the coefficients of the
first logarithmic terms in the volume expansions.
Gursky and the first author \cite{GurskyHan2017}, and Chen, Lai, and Wang \cite{ChenLaiWang2019}
studied Escobar's Yamabe compactifications for Poincar\'e-Einstein manifolds
by using the polyhomogeneous expansions of the singular Yamabe metrics.
In \cite{ShenWang2021}, the second and third authors obtained rigidity and gap results for boundary integrals of the global coefficient when $n=2$.
Recently, Li \cite{Li2022} recovered the existence of the singular Yamabe metric on compact manifolds with boundary through flow method. Chang, McKeown, and Yang \cite{CMY2022} studied the scattering problem for the singular Yamabe metrics and applied the results to the study of the conformal geometry of compact manifolds with boundary.

In \cite{HanShen3}, the first two authors studied the negativity of Ricci curvatures
of the singular Yamabe metrics and investigated whether these metrics
have negative sectional curvatures
or negative Ricci curvatures.
They proved that these metrics indeed have negative Ricci curvatures in bounded convex domains in the
Euclidean space.
They also provided general constructions of domains in compact manifolds and demonstrated that
the negativity of Ricci curvatures does not hold if the boundary is close to certain sets
of low dimensions. For singular Yamabe metrics, Ricci curvatures are
uniformly bounded from below if and only if they
are uniformly bounded since the scalar curvature
of these complete conformal metrics is assumed to be a fixed negative constant. %$-n(n-1)$.

In this paper, we study the blow-up phenomena of Ricci curvatures of the singular Yamabe metrics. In particular, we characterize the blow-up sets and use the blow-up sets to obtain some information of the conformal geometry of underlying manifold.
Our setting is similar to that in \cite{HanShen3} and will appear repeatedly.
It is convenient to formulate the following assumption.

\begin{assumption}\label{assumption-basic}
Let $(M,g)$ be a compact Riemannian manifold of dimension $n\geq 3$ without boundary and
$\Gamma\subset M$ be a nonempty compact set.
Suppose that $\Omega_i$ is a sequence of increasing smooth domains in $M\setminus\Gamma$,
which converges to $M \setminus\Gamma$, and that $g_i$ is the complete conformal metric in $\Omega_i$
with a constant scalar curvature $-n(n-1)$ as
in \cite{AM1988DUKE}.
\end{assumption}

\begin{figure}
\centering
\includegraphics[scale=0.15]{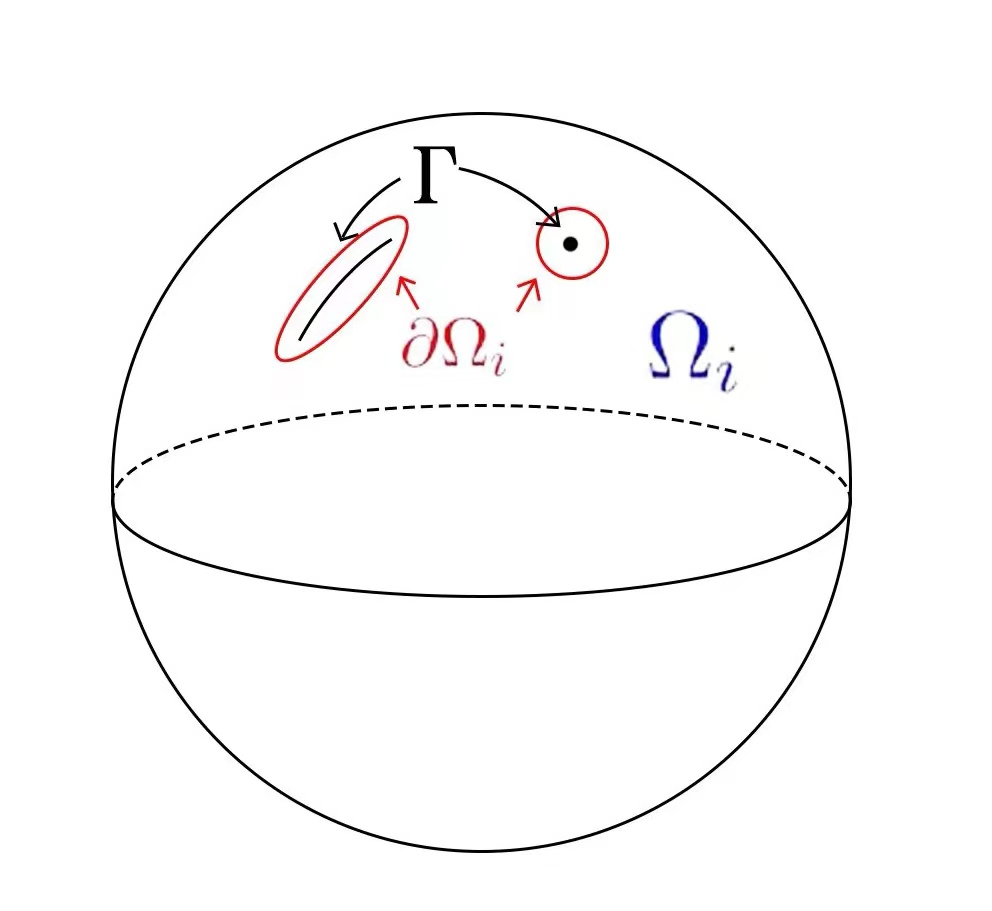}
\caption{$\Omega_i\rightarrow M\setminus\Gamma$}
\end{figure}

By the convergence of $\Omega_i$  to $M \setminus\Gamma$, we mean $\cup_{i=1}^\infty \Omega_i
=M\setminus \Gamma$ and, for any $\varepsilon>0$,
$\partial\Omega_i$ is in the $\varepsilon$-neighborhood of $\Gamma$ for all large $i$.
 By smooth domains, we always mean domains with smooth boundaries.
In \cite{HanShen3}, the first two authors proved that the maximal Ricci curvature of $g_i$ in $\Omega_{i}$
diverges to $\infty$ as $i\to\infty$, under appropriate assumptions on $(M,g)$ and $\Gamma$.
In particular, $\Gamma$ is required to be a disjoint union of finitely many closed smooth embedded submanifolds
in $M$ of varying dimensions,
between $0$ and $(n-2)/2$.

In the following, we refer to the compact set $\Gamma$ as the {\it limit set}.
With $\Gamma$, $\{\Omega_i\}$, and $\{g_i\}$ as in Assumption \ref{assumption-basic},
we define the associated {\it blow-up set} by
\begin{align*}
\mathcal{B}_{\Gamma,\{\Omega_i\}}
=\big\{x\in M;&\,\lim_{m\to\infty}\sup_{i}|Ric_{g_i}(x_m)|=\infty\\
&\text{ for some sequence }\{x_m\}\subset M\setminus\Gamma\text{ with }x_m \rightarrow x\big\}.\end{align*}
In other words, the blow-up set $\mathcal{B}_{\Gamma,\{\Omega_i\}}$
consists of the limits of sequences of points where the maximal Ricci components of $g_i$
diverges to $\infty$.
 It is obvious that the blow-up set $\mathcal{B}_{\Gamma,\{\Omega_i\}}$ is a closed set of $M$.
As we will see, the blow-up set $\mathcal{B}_{\Gamma,\{\Omega_i\}}$
depends on the limit set $\Gamma$ as well as the exhausting domains $\{\Omega_i\}$.
For a fixed compact set $\Gamma$ in $(S^n, g_{S^n})$, different sets of exhausting domains may yield different blow-up sets.

 In this paper, we aim to characterize the blow-up set $\mathcal{B}_{\Gamma,\{\Omega_i\}}$.
%and meanwhile relax requirements on $\Gamma$.
First, we assume that the limit $\Gamma$ consists of three components as follows:
\begin{align}\label{eq-def-Gamma1-2}\begin{split}
&\text{$\Gamma_{1}$ is a compact subset in $M$ of Hausdorff dimension less than $(n-2)/2$},\\
&\text{$\Gamma_{2}$ is a closed subset in
the union of countably many closed
$C^{1}$-embedded}\\
&\qquad\text{submanifolds in $M$ of dimension $(n-2)/2$},
\end{split}\end{align}
and
\begin{align}\label{eq-def-Gamma3}\begin{split}
&\text{$\Gamma_{3}$ is a disjoint union of finitely many closed smooth embedded}\\
&\qquad
\text{ submanifolds in $M$ of varying dimensions greater than $(n-2)/2$}.
\end{split}\end{align}
We point out that the lower dimensional component $\Gamma_1$ is allowed to have non-integer dimensions.
Even if it has an integer dimension, it is not necessarily a smooth submanifold.
We also point out that $\Gamma$ is allowed have a higher dimensional component $\Gamma_3$.
We emphasize that $\Gamma_2$ is assumed to be a closed set itself.

 We will prove the following two results. The higher dimensional part $\Gamma_3$ of the limit set $\Gamma$
is present in the first result and is absent in the second result.

\begin{theorem}\label{blow up set-nonempty-high}
Let $(M, g)$, $\Gamma$, and $\{\Omega_i\}$ be as in Assumption \ref{assumption-basic},
and $\Gamma=\Gamma_{1}\bigcup\Gamma_{2}\bigcup\Gamma_{3}$,
for $\Gamma_1, \Gamma_2$, and $\Gamma_3$ as in \eqref{eq-def-Gamma1-2} and \eqref{eq-def-Gamma3},
with $(\Gamma_{1}\bigcup\Gamma_{2})\bigcap\Gamma_3=\emptyset$,
$\Gamma_{1}\bigcup\Gamma_2\neq\emptyset$, and $\Gamma_{3}\neq\emptyset$.
Then, $\Gamma_{1}\bigcup\Gamma_{2}\subseteq \mathcal{B}_{\Gamma,\{\Omega_i\}}\subseteq \Gamma$.
\end{theorem}

\begin{theorem}\label{blow up set-empty-high}
Let $(M, g)$, $\Gamma$, and $\{\Omega_i\}$ be as in Assumption \ref{assumption-basic},
$\lambda(M,[g])$ be the Yamabe invariant of $(M, g)$,
and $\Gamma=\Gamma_{1}\bigcup\Gamma_{2}$,
for $\Gamma_1$ and $\Gamma_2$ as in \eqref{eq-def-Gamma1-2}.

$\mathrm{(1)}$ If $\lambda(M,[g])<0$, then $\mathcal{B}_{\Gamma,\{\Omega_i\}}=\Gamma$.

$\mathrm{(2)}$ If $\lambda(M,[g])=0$, then
$$ \mathcal{B}_{\Gamma,\{\Omega_i\}}=\Gamma \bigcup
\text{clos}\{x\in M:\,  Ric_{\bar{g}}(x)\neq0 \text{ for some }\bar{g}\in [g]\text{ with }R_{\bar{g}}=0\},
$$
where $R_{\bar{g}}$ is the scalar curvature of $\bar{g}$.

$\mathrm{(3)}$ If $\lambda(M,[g])>0$ and $(M, g)$ is not conformally equivalent to
$S^{n}$, then $\Gamma\subseteq \mathcal{B}_{\Gamma,\{\Omega_i\}}$.
If, in addition, $M$ is locally conformally flat, then $\mathcal{B}_{\Gamma,\{\Omega_i\}}=M$.

$\mathrm{(4)}$ If $(M, g)$ is conformally equivalent to $S^{n}$,
then either $\Gamma\subseteq \mathcal{B}_{\Gamma,\{\Omega_i\}}$
or $\Gamma\setminus\{p\}\subseteq \mathcal{B}_{\Gamma,\{\Omega_i\}}$
for some isolated point $p\in\Gamma$.
\end{theorem}

We point out that the second component in the union in (2) is a
well-defined closed set since any two conformal metrics of $g$
on $M$ with zero scalar curvatures differ from each other by a  positive constant factor
under the assumption $\lambda(M,[g])=0$.

 As we see in Theorem \ref{blow up set-empty-high}, the Yamabe invariant plays an essential role
in characterizing the blow-up sets.
We now make several comments.
First, (1)-(3) illustrate that the limit set $\Gamma$ is always contained in the blow-up set if
$(M, g)$ is not conformally equivalent to $S^{n}$.
Second, we will prove that all accumulation points of $\Gamma$ are always in the blow-up set.
The difference between (3) and (4)
lies on the number of isolated points in the limit set $\Gamma$ where blow-up occurs.
On manifolds not conformally
equivalent to the standard sphere, all isolated points in the limit set are in the blow-up set.
While on manifolds conformally equivalent to the standard sphere,
all but possibly one isolated point are in the blow-up set.
Third, the blow-up sets may contain points outside the limit sets $\Gamma$, and maybe the entire manifold,
as the second part of (3) demonstrates.

We will construct examples to demonstrate that Theorem  \ref{blow up set-empty-high}(4) is optimal.
For example, for the case $(M,g)=(S^n, g_{S^n})$ and $\Gamma=\{p,q\}$ for two distinct points $p,q\in S^n$,
with different sequences of exhausting domains,
the blow-up set can be any one of $p$ and $q$, or $\{p,q\}$, or even the entire $S^n$.
 For the case $(M,g)=(S^n, g_{S^n})$ and $\Gamma=\{p\}$ for a single point $p\in S^n$,
with different sequences of exhausting domains,
the blow-up set can be empty or nonempty.
In general, the blow-up set can be quite complex if the Yamabe invariant is positive.

The proofs of Theorems \ref{blow up set-nonempty-high}-\ref{blow up set-empty-high}
rely on a careful analysis of the Ricci curvatures of the complete conformal metrics near the boundary.
The Yamabe invariant of $(M,g)$ and the higher dimensional part $\Gamma_3$ play crucial roles and
determine behaviors of the convergence of the conformal factors.
As in \cite{HanShen3}, when the underlying manifold has a positive Yamabe invariant,
we need the positive mass theorem to study behaviors of conformal factors near isolated points in the limit set.

As consequences of Theorem \ref{blow up set-empty-high}, we have the following rigidity results.
Under certain conditions, if Ricci curvatures of a sequence of singular Yamabe metrics
remain bounded in a fixed small region, then the underlying manifold is conformally equivalent to
the standard sphere equipped with the standard spherical metric.

\begin{theorem}\label{them-rigidity}
Let $(M, g)$, $\Gamma$, $\{\Omega_i\}$, and $\{g_i\}$ be as in Assumption \ref{assumption-basic},
and $\Gamma=\Gamma_{1}\bigcup\Gamma_{2}$
for $\Gamma_1$ and $\Gamma_2$ as in \eqref{eq-def-Gamma1-2},
with an isolated point $x_0$ in $\Gamma$.
Suppose that the metric $g_i$
has uniformly bounded Ricci curvatures near $x_0$.
Then, $M$ is conformally equivalent to the standard sphere $S^n$.
\end{theorem}

\begin{theorem}\label{LCF-rigidity}
Let $(M, g)$, $\Gamma$, $\{\Omega_i\}$, and $\{g_i\}$ be as in Assumption \ref{assumption-basic},
$M$ be locally conformally flat with $\lambda(M,[g])>0$,
$\Gamma=\Gamma_{1}\bigcup\Gamma_{2}$
for $\Gamma_1$ and $\Gamma_2$ as in \eqref{eq-def-Gamma1-2},
and  $x_0$ be any point in $M$.
Suppose that the metric $g_i$ has uniformly bounded Ricci curvatures near $x_0$.
Then, $M$ is conformally equivalent to the standard sphere $S^n$.
\end{theorem}

The paper is organized as follows.
In Section \ref{An Equivalent Form},
we discuss some preliminary identities and present two important lemmas concerning conformal factors.
In Section \ref{sec-general result}, we
prove Theorem \ref{blow up set-nonempty-high} and
Theorem \ref{blow up set-empty-high}(1)-(2). We also prove that all accumulation points of the limit set
$\Gamma$ are in the blow-up sets.
In Section \ref{Isolated points}, we discuss isolated points of $\Gamma$ and
prove the first assertion in Theorem \ref{blow up set-empty-high}(3) and Theorem \ref{blow up set-empty-high}(4).
In Section \ref{sec-complement-Gamma}, we study whether blow-up sets contain points
outside the limit sets and prove the second assertion in Theorem \ref{blow up set-empty-high}(3).
In Section \ref{sec-examples}, we present several examples
in $S^{n}$ to demonstrate the complexity of blow-up sets.

We would like to thank Matthew Gursky, Marcus Khuri, and Yuguang Shi for helpful discussions.

\section{Preliminaries}\label{An Equivalent Form}

In this section,
we collect some basic results concerning the singular Yamabe metrics.
We also prove two important convergence results concerning the conformal factors of the singular Yamabe metrics.

Let $(M,g)$ be a smooth compact Riemannian manifold of dimension $n\ge 3$ without boundary.
The {\it Yamabe invariant} of $(M,g)$ is given by
$$\lambda(M, [g])=\inf\Big\{
\frac{\int_M  (  |\nabla_g \phi|^2 +\frac{n-2}{4(n-1)} S_g\phi^2 )dV_g   }{(\int_M \phi^{\frac{2n}{n-2}}dV_g)^{\frac{n-2}{n}}};
\,\phi \in C^{\infty}(M),\phi>0 \Big\}.$$
We define the conformal Laplacian of $(M,g)$ by
$$L_g=-\Delta_g +\frac{n-2}{4(n-1)}S_g.$$
For any functions $u$ and $\psi$ in $M$ with $\psi>0$, we have
\begin{equation}\label{confotmal laplacian}
L_g(\psi u)= \psi^{\frac{n+2}{n-2}}L_{\psi^{\frac{4}{n-2}}g}(u).
\end{equation}

Assume $\Omega\subset M$ is a smooth domain,
with an $(n-1)$-dimensional boundary. We consider the following problem:
\begin{align}
\label{eq-MEq} \Delta_{g} u -\frac{n-2}{4(n-1)}S_gu&= \frac14n(n-2) u^{\frac{n+2}{n-2}} \quad\text{in }\,\Omega,\\
\label{eq-MBoundary}u&=\infty\quad\text{on }\partial \Omega,
\end{align}
where $S_g$ is the scalar curvature of $M$.
According to Loewner and Nirenberg \cite{Loewner&Nirenberg1974} for $(M,g)=(S^n, g_{S^n})$ and
Aviles and McOwen \cite{AM1988DUKE} for the general case,
\eqref{eq-MEq} and \eqref{eq-MBoundary} admit a unique positive solution.
The singular Yamabe metric $u^{\frac{4}{n-2}}g$ is
a complete metric with a constant scalar curvature $-n(n-1)$ on $\Omega$.

\smallskip

Concerning boundary behaviors of the conformal factors,
Andersson, Chru\'sciel, and Friedrich \cite{ACF1982CMP} and Mazzeo \cite{Mazzeo1991} established
the polyhomogeneous expansions. For the first several terms, we have
$$u=d^{-\frac{n-2}{2}}\Big[1+\frac{n-2}{4(n-1)}H_{\partial\Omega}d+O(d^2)\Big],$$
where $d$ is the distance to $\partial\Omega$ and $H_{\partial\Omega}$ is the mean curvature of $\partial\Omega$
with respect to the interior
unit normal vector of $\partial\Omega$.

Set
\begin{equation*}%\label{eq-def-v}
v=u^{-\frac{2}{n-2}}.\end{equation*}
Then,
\begin{align}\label{eq-MEq-v}
v\Delta_{g} v +\frac{1}{2(n-1)}S_{ g} v^2&= \frac{n}{2}(|\nabla_{g} v|^2-1)\quad\text{in }\Omega, \\
\label{eq-MBoundary-v}v&=0\quad\text{on }\partial \Omega.
\end{align}
Moreover,
\begin{equation}\label{boundary expansion v}v=d-\frac{1}{2(n-1)}H_{\partial\Omega}d^2+O(d^3).\end{equation}
This implies
\begin{equation}\label{boundary expansion v-gradient}|\nabla_gv|=1\quad\text{on }\partial\Omega.\end{equation}

Consider the conformal metric
\begin{equation*}%\label{eq-conformal-metric}
g_\Omega=u^{\frac{4}{n-2}}g=v^{-2}g.\end{equation*}
For a unit vector $X$ with respect to $g$, $vX$ is a unit vector with respect to $g_\Omega$.
Let $R_{ij}$ be the Ricci components of $g$ in a local frame for the metric $g$ and
$R^\Omega_{ij}$ be the Ricci components of $g_\Omega$
in the corresponding frame for the metric $g_\Omega$.
Then,
\begin{equation}\label{Ricci cur  in v Manifold0}
{R}^\Omega_{kl}=v^2R_{kl}+(n-2)\big[vv_{,kl}-\frac12g_{kl}|\nabla _gv|^2\big]
+g_{kl}\big[v\Delta_gv-\frac{n}{2}|\nabla_gv|^2\big].
\end{equation}
By \eqref{eq-MEq-v}, we have
\begin{equation}\label{Ricci cur  in v Manifold}
{R}^\Omega_{kl}=v^2R_{kl}-\frac{1}{2(n-1)}v^2g_{kl}S_g+(n-2)\big[vv_{,kl}-\frac{1}{2}g_{kl}|\nabla_gv|^2\big]-\frac{n}{2}g_{kl}.
\end{equation}

We now present two convergence results which generalize Lemma 4.1 in \cite{HanShen3}.
They play an important role in this paper.

\begin{lemma}\label{lemme-convergence1}
Let $(M, g)$, $\Gamma$, and $\{\Omega_i\}$ be as in Assumption \ref{assumption-basic},
and $\Gamma=\Gamma_{1}\bigcup\Gamma_{2}$
for $\Gamma_1$ and $\Gamma_2$ as in \eqref{eq-def-Gamma1-2}.
Suppose that the scalar curvature $S_g$ is constant and
that $u_i$ is the solution of \eqref{eq-MEq} and \eqref{eq-MBoundary} in $\Omega_i$.
Then, for any positive integer $m$, if $S_g\ge 0$,
\begin{equation}\label{u_i go to zero-plus-v1}
u_i\rightarrow 0\quad\text{in }C^{m}_{\mathrm{loc}}( M \setminus\Gamma)
\text{ as }i\rightarrow\infty,\end{equation}
and, if $S_g<0$,
\begin{equation}\label{u_i go to zero-minus}
u_i\rightarrow \Big(\frac{-S_g}{n(n-1)}\Big)^{\frac{n-2}{4}}
\quad\text{in }C^{m}_{\mathrm{loc}}( M \setminus\Gamma)
\text{ as }i\rightarrow\infty.\end{equation}
\end{lemma}

\begin{proof} Recall that $\Gamma_{1}$ is a compact set in $M$ of Hausdorff dimension less than $(n-2)/2$
and $\Gamma_{2}$ is a closed subset in the union $\bigcup_{l=1}^{\infty}N_{l}$, where each $N_{l}$ is a closed
$C^{1}$-embedded submanifold in $M$ of dimension $(n-2)/2.$

As in the proof of Lemma 4.1 in \cite{HanShen3},
we have $u_i \geq u_{i+1}$ in $\Omega_i$. It is straightforward to verify, for any $m$,
$$u_i\rightarrow u\quad\text{in }C^{m}_{\mathrm{loc}}( M \setminus\Gamma)
\text{ as }i\rightarrow\infty,$$
where $u$ is a nonnegative solution of \eqref{eq-MEq} in $M \setminus\Gamma$.
We consider two cases.

{\it Special Case.}
We first assume $\Gamma_2=\bigcup_{l=1}^{k}N_l$, where $k$ is a finite integer.
Moreover, $\Gamma_1$ and $N_l$ are disjoint.

By Theorem 2.1 and Theorem 2.2 in \cite{A1982CPDE},
$u$ can be extended to the whole manifold $M$ and \eqref{eq-MEq} holds in $M.$
Then,
\begin{equation}\label{u indentity}
- \int_{M}|\nabla_g u|^2 dV_g=  \int_{M}\Big( \frac{n-2}{4(n-1)}S_gu^2+ \frac14n(n-2) u^{\frac{2n}{n-2}} \Big)dV_g.
\end{equation}
If $S_g\ge 0$, then $u=0$ in $M$.
If $S_g<0$, by the proof of Lemma 4.1 in \cite{HanShen3}, we have
\begin{equation*}%\label{eq-AlgebraicRelation}
u_i\geq
\Big(\frac{-S_g}{n(n-1)}\Big)^{\frac{n-2}{4}}
\quad\text{in }\Omega_i.\end{equation*}
Hence,
\begin{equation*}u\geq
\Big(\frac{-S_g}{n(n-1)}\Big)^{\frac{n-2}{4}}
\quad\text{in }M.\end{equation*}
Combining with \eqref{u indentity}, we have $u=\Big(\frac{-S_g}{n(n-1)}\Big)^{\frac{n-2}{4}}$ in $M$.

{\it General Case.} We now consider the general case.
For unification, we write  $N_0=\Gamma_1$.
For each $l\geq0$,
let $\Omega_{i}^{l}$ be a sequence of increasing smooth domains in $M \setminus N_{l}$,
which converges to $M \setminus N_{l}$,
and $u_{i}^{l}$ be the solution
of \eqref{eq-MEq} and \eqref{eq-MBoundary} in $\Omega_{i}^{l}$.

First, we consider $S_g\geq0$.
Take a fixed point $p\in M\setminus\overline{\cup_{l=0}^{\infty}N_{l}}$,
and we will prove $u(p)=0$. Then, we have $u=0$ in $M\setminus \Gamma$.
As in the special case, for any fixed $l\geq0$,
we have $u_{i}^{l}\rightarrow0$
uniformly in any compact subsets of $M\setminus N_l$
as $i\rightarrow\infty$. Then, for any $\epsilon>0$, there exists an integer $i(l,\epsilon)$
such that $u_{i(l,\epsilon)}^l(p)<\epsilon/2^{l+1}$.
Note that both $\Gamma_1$ and $\Gamma_2$ are closed
and $\bigcup_{l=0}^{\infty}M\setminus \overline{\Omega_{i(l,\epsilon)}^{l} }$
covers $\Gamma_1\bigcup\Gamma_2$.
Then, $\Gamma_1\bigcup\Gamma_2$ has a finite subcover, given by
$M\setminus \overline{\Omega_{i(0,\epsilon)}^{0} }$,
$M\setminus \overline{\Omega_{i(l_1,\epsilon)}^{l_1} }$, $\cdots$,
$M\setminus \overline{\Omega_{i(l_k,\epsilon)}^{l_k} }$.
It is easy to verify $u_{i(0,\epsilon)}^0+\sum_{j=1}^ku_{i(l_j,\epsilon)}^{l_j}$ is a supersolution of \eqref{eq-MEq}
in $\Omega_{i(0,\epsilon)}^0\bigcap\big(\bigcap_{j=1}^{k}\Omega_{i(l_j,\epsilon)}^{l_j}\big)$.
Then by the maximum principle, we have $u(p)<\epsilon$. Hence, $u(p)=0$.

Next, we consider $S_g<0$. Without loss of generality, we assume $S_g=-n(n-1)$.
Take a fixed point $p\in M\setminus\overline{\cup_{l=0}^{\infty}N_{l}}$, and we will prove $u(p)=1$.
Then, we have $u=1$ in $M\setminus \Gamma$.
By the special case, for any fixed $l\geq0$,
we have $u_{i}^{l}\geq 1$ and $u_{i}^{l}\rightarrow 1$
uniformly in any compact subsets of $M\setminus N_l$,
as $i\rightarrow\infty$. Then, for any $\epsilon>0$, there exists  an integer  $i(l,\epsilon)$
such that $u_{i(l,\epsilon)}^l(p)<1+\epsilon/2^{l+1}$.
Note that both $\Gamma_1$ and $\Gamma_2$ are closed and
$\bigcup_{l=0}^{\infty}M\setminus \overline{\Omega_{i(l,\epsilon)}^{l} }$ covers $\Gamma_1\bigcup\Gamma_2$.
Then, $\Gamma_1\bigcup\Gamma_2$ has a finite subcover, given by
$M\setminus \overline{\Omega_{i(0,\epsilon)}^{0} }$,
$M\setminus \overline{\Omega_{i(l_1,\epsilon)}^{l_1} }$, $\cdots$, $M\setminus \overline{\Omega_{i(l_k,\epsilon)}^{l_k} }$.
We can verify that
$$1+(u_{i(0,\epsilon)}^0-1)+\sum_{j=1}^k(u_{i(l_j,\epsilon)}^{l_j}-1)$$
is a supersolution of \eqref{eq-MEq}
in $\Omega_{i(0,\epsilon)}^0\bigcap\big(\bigcap_{j=1}^{k}\Omega_{i(l_j,\epsilon)}^{l_j}\big)$.
Then by the maximum principle, we have $u(p)<1+\epsilon$.
On the other hand, we have $u\geq 1$ as in the special case. Hence, $u(p)=1$.
\end{proof}

If $\Gamma_3$ is not an empty set, we have the following convergence result.

\begin{lemma}\label{lemme-convergence2}
 Let $(M, g)$, $\Gamma$, and $\{\Omega_i\}$ be as in Assumption \ref{assumption-basic},
and $\Gamma=\Gamma_{1}\bigcup\Gamma_{2}\bigcup\Gamma_{3}$
for $\Gamma_1, \Gamma_2$, and $\Gamma_3$ as in \eqref{eq-def-Gamma1-2} and \eqref{eq-def-Gamma3},
with $(\Gamma_{1}\bigcup\Gamma_{2})\bigcap\Gamma_3=\emptyset$ and $\Gamma_{3}\neq\emptyset$.
Suppose that the  scalar curvature $S_g$ is constant and
that  $u_i$ is the solution
of \eqref{eq-MEq} and \eqref{eq-MBoundary} in $\Omega_i$.
Then, for any positive integer $m$,
\begin{equation}\label{u_i go to zero-plus}
u_i\rightarrow u\quad\text{in }C^{m}_{\mathrm{loc}}( M \setminus\Gamma)
\text{ as }i\rightarrow\infty,\end{equation}
for some $u\in C^{\infty}( M \setminus\Gamma_{3}) $ satisfying $u>0$ and
$$\Delta_{g} u -\frac{n-2}{4(n-1)}S_gu= \frac14n(n-2) u^{\frac{n+2}{n-2}}
\quad\text{in }M  \setminus\Gamma_{3}.$$
\end{lemma}

\begin{proof}
By the maximum principle, we have $u_i \geq u_{i+1}$ in $\Omega_i$. It is straightforward to verify, for any $m$,
$$u_i\rightarrow u\quad\text{in }C^{m}_{\mathrm{loc}}( M \setminus\Gamma)
\text{ as }i\rightarrow\infty,$$
where $u$ is a nonnegative solution of \eqref{eq-MEq} in $M \setminus\Gamma$.

According to \cite{AM1988DUKE}, there exists a complete conformal metric $\underline{u}^{\frac{4}{n-2}}g$
with a negative constant scalar curvature $-n(n-1)$ on $M  \setminus\Gamma_{3}$.
Although results in \cite{AM1988DUKE} were formulated for $\Gamma_3$ as a smooth submanifold,
their proof also holds for $\Gamma_3$ as in \eqref{eq-def-Gamma3}.
Hence, \eqref{eq-MEq} and \eqref{eq-MBoundary} have a positive solution $\underline{u}$
for $\Omega=M\setminus\Gamma_3$.

Take an open set $\Omega'$ in $M$ such that
$\Gamma_1\bigcup\Gamma_2\subset \Omega'\subset\subset M\setminus \Gamma_3$.
Set $\Omega_{i3}=\Omega_i\bigcup \Omega'$.
Then, $\Omega_{i3}\rightarrow M\setminus \Gamma_3$, and $\Omega_i\subseteq\Omega_{i3}$  for $i$ large.
Let $u_{i3}$ be the solution
of \eqref{eq-MEq} and \eqref{eq-MBoundary} in $\Omega_{i3}$.
By the maximum principle, we have
$u_{i3}>\underline{u}>0$ in  $\Omega_{i3}$. Then, for any $m$,
$$u_{i3}\rightarrow \widetilde{u}\quad\text{in }C^{m}_{\mathrm{loc}}( M \setminus\Gamma_3)
\text{ as }i\rightarrow\infty,$$
where $\widetilde{u}$ is a nonnegative solution of \eqref{eq-MEq} in $M \setminus\Gamma_3$.
Moreover, \begin{equation}\label{u widetildeu underlineu}u\geq\widetilde{u}\geq\underline{u}\end{equation}
 in $M \setminus\Gamma$.

For unification, we write  $N_0=\Gamma_1$.
For each $l\geq0$,
let $\Omega_{i}^{l}$ be a sequence of increasing smooth domains in $M \setminus N_{l}$,
which converges to $M \setminus N_{l}$,
and $u_{i}^{l}$ be the solution
of \eqref{eq-MEq} and \eqref{eq-MBoundary} in $\Omega_{i}^{l}$.

Take a fixed point $p\in M\setminus\overline{ \big(\bigcup(\bigcup_{l=0}^{\infty}N_{l})\bigcup\Gamma_3\big)}$,
and we will prove $u(p)=\widetilde{u}(p)$.
Then, we have $u=\widetilde{u}$ in $M\setminus \Gamma$.

First, we consider $S_g\geq0$.
By Lemma \ref{lemme-convergence1}, for any fixed $l\geq0$,
we have $u_{i}^{l}\rightarrow0$
uniformly in any compact subsets of $M\setminus N_l$
as $i\rightarrow\infty$. Then, for any $\epsilon>0$, there exist  an integer
$i(l,\epsilon)$ such that $u_{i(l,\epsilon)}^l(p)<\epsilon/2^{l+1}$ and an integer $i(\epsilon)$
such that $u_{i(\epsilon)3}(p)<\widetilde{u}(p)+\epsilon$.
Note that both $\Gamma_1$ and $\Gamma_2$ are closed
and $\bigcup_{l=0}^{\infty}M\setminus \overline{\Omega_{i(l,\epsilon)}^{l} }$ covers $\Gamma_1\bigcup\Gamma_2$.
Then, $\Gamma_1\bigcup\Gamma_2$ has a finite subcover, given by
$M\setminus \overline{\Omega_{i(0,\epsilon)}^{0} }$,
$M\setminus \overline{\Omega_{i(l_1,\epsilon)}^{l_1} }$,
$\cdots$, and $M\setminus \overline{\Omega_{i(l_k,\epsilon)}^{l_k} }$.
It is easy to verify that
$$u_{i(0,\epsilon)}^0+\sum_{j=1}^ku_{i(l_j,\epsilon)}^{l_j}+u_{i(\epsilon)3}$$
is a supersolution of \eqref{eq-MEq}
in $\Omega_{i(0,\epsilon)}^0\bigcap\big(\bigcap_{j=1}^{k}\Omega_{i(l_j,\epsilon)}^{l_j}\big)\bigcap\Omega_{i(\epsilon)3}$.
Then by the maximum principle, we have $u(p)<\widetilde{u}(p)+2\epsilon$.
Combining with \eqref{u widetildeu underlineu}, we have $u(p)=\widetilde{u}(p)$.

Next, we consider $S_g<0$. Without loss of generality, we assume $S_g=-n(n-1)$.
By Lemma \ref{lemme-convergence1}, for any fixed $l\geq0$,
we have $u_{i}^{l}\geq 1$ and $u_{i}^{l}\rightarrow1$
uniformly in any compact subsets of $M\setminus N_l$
as $i\rightarrow\infty$. We also have $\widetilde{u}\geq 1$.
Then, for any $\epsilon>0$, there exist an integer $i(l,\epsilon)$ such that
$u_{i(l,\epsilon)}^l(p)<1+\epsilon/2^{l+1}$ and an integer $i(\epsilon)$ such that
$u_{i(\epsilon)3}(p)<\widetilde{u}(p)+\epsilon$.
Note that both $\Gamma_1$ and $\Gamma_2$ are closed
and $\bigcup_{l=0}^{\infty}M\setminus \overline{\Omega_{i(l,\epsilon)}^{l} }$ covers $\Gamma_1\bigcup\Gamma_2$.
Then, $\Gamma_1\bigcup\Gamma_2$ has a finite subcover, given by
$M\setminus \overline{\Omega_{i(0,\epsilon)}^{0} }$,
$M\setminus \overline{\Omega_{i(l_1,\epsilon)}^{l_1} }$,
$\cdots$, $M\setminus \overline{\Omega_{i(l_k,\epsilon)}^{l_k} }$.
We can verify that
$$(u_{i(0,\epsilon)}^0-1)+\sum_{j=1}^k(u_{i(l_j,\epsilon)}^{l_j}-1)+u_{i(\epsilon)3}$$
is a supersolution of \eqref{eq-MEq}
in $\Omega_{i(0,\epsilon)}^0\bigcap\big(\bigcap_{j=1}^{k}\Omega_{i(l_j,\epsilon)}^{l_j}\big)\bigcap\Omega_{i(\epsilon)3}$.
Then by the maximum principle, we have $u(p)<\widetilde{u}(p)+2\epsilon$.
Combining with \eqref{u widetildeu underlineu}, we have $u(p)=\widetilde{u}(p)$.
\end{proof}

\section{General Blow-up Results}\label{sec-general result}

In this section,  we study blow-up phenomena if the limit set $\Gamma$ has a compact set of a lower Hausdorff dimension.
We will prove
Theorem \ref{blow up set-nonempty-high} and
Theorem \ref{blow up set-empty-high}(1)-(2). We also prove that all accumulation points of the limit set
$\Gamma$ are in the blow-up sets.

 We first prove some preliminary results.
Let $(M, g)$, $\Gamma$, $\{\Omega_i\}$, and $\{g_i\}$ be as in Assumption \ref{assumption-basic}.
For each $i$, consider the solution $u_i$ of
\begin{align}
\label{eq-MEq-i} \Delta_{g} u_i -\frac{n-2}{4(n-1)}S_gu_i&= \frac14n(n-2) u_i^{\frac{n+2}{n-2}} \quad\text{in }\,\Omega_i,\\
\label{eq-MBoundary-i}u_i&=\infty\quad\text{on }\partial \Omega_i,
\end{align}
and set
\begin{equation}\label{eq-def-v-i}
v_i=u_i^{-\frac{2}{n-2}}.\end{equation}
Then,
$$g_{i}=u_i^{\frac{4}{n-2}}g=v_i^{-2}g.$$
By \eqref{Ricci cur  in v Manifold}, the Ricci curvature of $g_i$ is given by
\begin{equation}\label{Ricci-cur-in-v-Manifold}
{R}^i_{kl}=v_i^2\Big[R_{kl}-\frac{1}{2(n-1)}g_{kl}S_g\Big]+(n-2)v_iv_{i,kl}-\frac{n-2}{2}g_{kl}|\nabla_gv_i|^2-\frac{n}{2}g_{kl}.
\end{equation}
To study whether $R^i_{kl}$ blows up or stays bounded, it is important to estimate $\nabla^2_g v_i$, related to
$v_i$ and $\nabla_gv_i$.

\begin{lemma}\label{lemma-blowup-set1}
Let $(M, g)$, $\Gamma$, $\{\Omega_i\}$, and $\{g_i\}$ be as in Assumption \ref{assumption-basic},
$u_i$ and $v_i$ as in \eqref{eq-MEq-i}-\eqref{eq-MBoundary-i} and \eqref{eq-def-v-i}, and $x_0\in \Gamma$.
For some $r>0$, assume, for any $m$,
\begin{equation}\label{eq-v convergence-v0}
v_i \rightarrow
v \quad\text{in }C^{m}_{\mathrm{loc}}(B_r(x_0) \setminus\Gamma)
\text{ as }i\rightarrow\infty,\end{equation}
for some positive smooth function $v$  in $B_r(x_0)$.
Then, $x_0\in \mathcal{B}_{\Gamma,\{\Omega_i\}}$.
\end{lemma}

We point out that the convergence in \eqref{eq-v convergence-v0} is away from $\Gamma$
but $v$ is a positive smooth function in the entire $B_r(x_0)$.

\begin{proof}
For any sufficiently small $\epsilon>0$, we choose normal coordinates in a small neighborhood of
$x_0$ such that $x_0=0$
and the line segment $\{te_n;\, t\in [0,\epsilon]\}$ on the $x_n$-axis is a geodesic connecting $x_0$
and $\epsilon e_n\in \Omega_i$ for $i$ large. Here, $e_n=(0,...,0,1)$.
Take $t^*_i\in (0,\epsilon)$ such that $t^*_ie_n\in\partial\Omega_i$ and
$te_n \in \Omega_i$ for any $t\in (t_i^*, \epsilon]$.
By the polyhomogeneous expansion of $v_i$, we have
$$|\partial_n v_i(t^*_ie_n)|\leq 1.$$
By \eqref{eq-v convergence-v0},
\begin{equation}\label{eq-compare-vi}|\partial_nv_i(\epsilon e_n )|\leq |\partial_nv(\epsilon e_n )|+1\quad\text{for $i$ large}.
\end{equation}
For $\epsilon$ small and $i$ large, we take ${t}_i \in [t^*_i,\epsilon]$ such that, for any $t\in [t^*_i,\epsilon]$,
$$\partial_nv_i(te_n)\leq \partial_nv_i({t}_ie_n).$$
By $v_i(t^*_ie_n)=0$ and $v_i(\epsilon e_n)\to v(\epsilon e_n)$ as $i\to\infty$, we have
$$\partial_nv_i({t}_ie_n)>\frac{v_i(\epsilon e_n)-0}{\epsilon-t^*_i}>
\frac{1}{2\epsilon}v(\epsilon e_n).$$
Note that $v$ is smooth and positive near $x_0=0$.
In particular, $v(\epsilon e_n)$ has a positive lower bound independent of $\epsilon$.
Thus, ${t}_i \in (t^*_i,\epsilon)$, and hence
$$\partial_{nn}v_i({ t}_ie_n)=0.$$
Denote by $R^i_{nn}$ the Ricci curvature of $g_i$
acting on the unit vector $v_i \partial_n$ with respect to the metric $g_i$.
Substituting into \eqref{Ricci-cur-in-v-Manifold}, at the point ${t}_ie_n$, we can verify
$R_{nn}^i\le -C\epsilon^{-2}$, for all large $i$ and some positive constant $C$
independent of $i$ and $\epsilon$.
By choosing a sequence $\epsilon_k\to 0$,
we have $x_0\in \mathcal{B}_{\Gamma,\{\Omega_i\}}$.
\end{proof}

 More generally, we have the following result.

\begin{lemma}\label{lemma-blowup-set2}
Let $(M, g)$, $\Gamma$, $\{\Omega_i\}$, and $\{g_i\}$ be as in Assumption \ref{assumption-basic},
$u_i$ and $v_i$ as in \eqref{eq-MEq-i}-\eqref{eq-MBoundary-i} and \eqref{eq-def-v-i}, and $x_0\in \Gamma$.
For any sufficiently small $\varepsilon>0$ and any large $i$,
assume there exist a point $x_{\varepsilon i}\in B_{\varepsilon}(x_0)\cap \Omega_i$
and a unit vector $\nu_{\varepsilon i}$ (with respect to the metric $g$)
such that
\begin{equation}\label{eq-v_i-first-derivative-infty0}
|\nabla_gv_i(x_{\varepsilon i})|\to\infty  \quad\text{as }i\rightarrow\infty,\end{equation}
and
\begin{equation}\label{eq-v_icontrol0}
v_i(x_{\varepsilon i})+|\partial_{\nu_{\varepsilon i}\nu_{\varepsilon i}}v_i(x_{\varepsilon i})|
\le \varepsilon |\nabla_gv_i(x_{\varepsilon i})|.\end{equation}
Then, $x_0\in \mathcal{B}_{\Gamma,\{\Omega_i\}}$.
\end{lemma}

\begin{proof}
 Denote by $R^i_{\nu_{\varepsilon i}\nu_{\varepsilon i}}$ the Ricci curvature of $g_i$
acting on the unit vector ${v_i\nu_{\varepsilon i}}$ with respect to the metric $g_i$.
By substituting \eqref{eq-v_icontrol0} in \eqref{Ricci-cur-in-v-Manifold}
and making use of \eqref{eq-v_i-first-derivative-infty0}, we conclude that
the Ricci component $R^i_{\nu_{\varepsilon i}\nu_{\varepsilon i}}$ at $x_{\varepsilon i}$
diverges to $ -\infty$ as $i\rightarrow\infty$. Hence,  some Ricci components of $g_i$
at $x_{\varepsilon i}$ diverge to $\infty$ as $i\rightarrow\infty$.
By choosing a sequence $\epsilon_k\to 0$,
we have $x_0\in \mathcal{B}_{\Gamma,\{\Omega_i\}}$.
\end{proof}

 We are ready to study the blow-up sets.
We first consider the case that $\Gamma$ has a nonempty higher dimensional part.

\begin{theorem}\label{gemma3 nonempty}
Let $(M, g)$, $\Gamma$, and $\{\Omega_i\}$ be as in Assumption \ref{assumption-basic},
and $\Gamma=\Gamma_{1}\bigcup\Gamma_{2}\bigcup\Gamma_{3}$
for $\Gamma_1, \Gamma_2$, and $\Gamma_3$ as in \eqref{eq-def-Gamma1-2} and \eqref{eq-def-Gamma3},
with $(\Gamma_{1}\bigcup\Gamma_{2})\bigcap\Gamma_3=\emptyset$,
$\Gamma_{1}\bigcup\Gamma_2\neq\emptyset$, and $\Gamma_{3}\neq\emptyset$.
Then, $\Gamma_{1}\bigcup\Gamma_{2}\subseteq \mathcal{B}_{\Gamma,\{\Omega_i\}}\subseteq \Gamma$.
\end{theorem}

\begin{proof}
Let  $u_i$ be the solution
of \eqref{eq-MEq-i} and \eqref{eq-MBoundary-i} in $\Omega_i$ and set $v_i$ by \eqref{eq-def-v-i}.
By Lemma \ref{lemme-convergence2}, for any $m$,
$$u_i\rightarrow u
\quad\text{in }C^{m}_{\mathrm{loc}}( M \setminus\Gamma)
\text{ as }i\rightarrow\infty,$$
where $u$ is a positive smooth function in $M \setminus\Gamma_3$. Hence,
\begin{equation}\label{eq-v convergence-v1}
v_i \rightarrow
v:=u^{-\frac{2}{n-2}} \quad\text{in }C^{m}_{\mathrm{loc}}( M \setminus\Gamma)
\text{ as }i\rightarrow\infty.\end{equation}
Note that $v$ is a positive smooth function in $M \setminus\Gamma_3$, and,
in particular, near $\Gamma_1\bigcup\Gamma_2$.
By \eqref{Ricci-cur-in-v-Manifold} and \eqref{eq-v convergence-v1},
it is straightforward to verify
that the Ricci curvature of $g_i$ remains bounded uniformly for large $i$
near any point in $M\setminus\Gamma$; namely, $\mathcal{B}_{\Gamma,\{\Omega_i\}}\subseteq\Gamma$.
 By Lemma \ref{lemma-blowup-set1}, we have $x_0\in \mathcal{B}_{\Gamma,\{\Omega_i\}}$ for any
$x_0\in \Gamma_{1}\bigcup\Gamma_{2}$,
and hence $\Gamma_{1}\bigcup\Gamma_{2}\subseteq \mathcal{B}_{\Gamma,\{\Omega_i\}}$.
\end{proof}

We next study the case that the higher dimensional part is absent from $\Gamma$, i.e., $\Gamma_3=\emptyset$.
We first consider the case that the Yamabe invariant $\lambda(M,[g])$ of $(M, g)$ is nonpositive.
For the case $\lambda(M,[g])=0$,  set
\begin{equation}\label{eq-definition-S-NRF}
\mathcal{S}_g=\text{clos}\{x|x\in M,  Ric_{\bar{g}}(x)\neq0 \text{ for some }\bar{g}\in [g]\text{ with }R_{\bar{g}}=0\},
\end{equation}
where $R_{\bar{g}}$ is the scalar curvature of $\bar{g}$.
The set  $\mathcal{S}_g$ is closed and  well-defined since any two conformal metrics of $g$
on $M$ with zero scalar curvatures differ from each other by a  positive constant factor.

\begin{theorem}\label{nonpositive yam inv}
Let $(M, g)$, $\Gamma$, and $\{\Omega_i\}$ be as in Assumption \ref{assumption-basic},
$\lambda(M,[g])$ be the Yamabe invariant of $(M, g)$,
and $\Gamma=\Gamma_{1}\bigcup\Gamma_{2}$
for $\Gamma_1$ and $\Gamma_2$ as in \eqref{eq-def-Gamma1-2}.

$\mathrm{(1)}$ If $\lambda(M,[g])<0$, then $\mathcal{B}_{\Gamma,\{\Omega_i\}}=\Gamma$.

$\mathrm{(2)}$ If $\lambda(M,[g])=0$, then
$\mathcal{B}_{\Gamma,\{\Omega_i\}}=\Gamma \bigcup \mathcal{S}_g$.
\end{theorem}

\begin{proof}
By the solution of the Yamabe problem \cite{LeeParker1987}, we can assume the scalar curvature $S_g$
of $M$ is the constant $\lambda(M,[g])$.
Since $M$ is compact, we take $\Lambda>0$ such that
$$|R_{ij}|\leq \Lambda g_{ij}.$$
Let  $u_i$ be the solution
of \eqref{eq-MEq-i} and \eqref{eq-MBoundary-i} in $\Omega_i$ and set $v_i$ by \eqref{eq-def-v-i}.
We now discuss two cases, $S_g<0$ and $S_g=0$.

\smallskip
{\it Case 1}: $S_g<0$.
By Lemma \ref{lemme-convergence1}, for any $m$,
$$u_i\rightarrow \Big(\frac{-S_g}{n(n-1)}\Big)^{\frac{n-2}{4}}
\quad\text{in }C^{m}_{\mathrm{loc}}( M \setminus\Gamma)
\text{ as }i\rightarrow\infty,$$
and hence
\begin{equation}\label{eq-v convergence-v2}
v_i \rightarrow
\Big(\frac{-S_g }{n(n-1)}\Big)^{-\frac{1}{2}} \quad\text{in }C^{m}_{\mathrm{loc}}( M \setminus\Gamma)
\text{ as }i\rightarrow\infty.\end{equation}
By \eqref{Ricci-cur-in-v-Manifold} and \eqref{eq-v convergence-v2},
it is straightforward to verify $\mathcal{B}_{\Gamma,\{\Omega_i\}}\subseteq\Gamma$.
We can proceed as in the proof of Theorem \ref{gemma3 nonempty}
to obtain $\Gamma\subseteq \mathcal{B}_{\Gamma,\{\Omega_i\}}$.
Therefore,
$\mathcal{B}_{\Gamma,\{\Omega_i\}}=\Gamma.$

\smallskip

{\it Case 2}:  $S_g=0$. In this case, we have
for any $m$,
$$u_i\rightarrow 0\quad\text{in }C^{m}_{\mathrm{loc}}( M \setminus\Gamma)
\text{ as }i\rightarrow\infty,$$
and hence
\begin{equation*}
v_i\rightarrow \infty
\quad\text{uniformly in any compact subsets of $M \setminus\Gamma$
as }i\rightarrow\infty.\end{equation*}
By \eqref{eq-MEq-v}-\eqref{eq-MBoundary-v}, we have
\begin{align*}
\Delta_{g} v_i^2 &=(n+2)|\nabla_{g} v_i|^2-n\quad\text{in }\Omega_i, \\
v_i&=0\quad\text{on }\partial \Omega_i.
\end{align*}
Then,
\begin{equation*}%\label{gead v indentity}
\int_{\Omega_i}|\nabla_{g} v_i|^2dV_{g}=\frac{n}{n+2}\int_{\Omega_i}dV_{g}<\frac{n}{n+2}\int_{M}dV_{g},\end{equation*}
and hence
\begin{equation*}%\label{gead v upbound}
\int_{\Omega_i} u_i^{-\frac{2n}{n-2}}|\nabla_{g} u_i|^2dV_{g}<
\frac{n(n-2)^2}{4(n+2)}\int_{M}dV_{g}.\end{equation*}

Take any $\Omega''\subset\subset\Omega'\subset\subset M\setminus\Gamma$.
For $i$ sufficiently large, $\Omega'\subset\subset\Omega_i$.
By the Hanarck inequality, we have
\begin{equation*}
\sup_{\Omega'}u_i \leq C(\Omega') \inf_{\Omega'}u_i,\end{equation*}
where $C(\Omega')$ is a positive constant independent of $i$. Hence, fixing $p\in \Omega'$, we have
\begin{equation*}
\|\nabla_{g} u_i \|_{L^{2}(\Omega')}\leq C(\Omega')u_i^{\frac{n}{n-2}}(p).
\end{equation*}
For any $x\in \Omega''$, choosing normal coordinates near $x$, we have
\begin{equation*}\partial_{m}\big[\partial_k(\sqrt{\det g_{kl}}g^{kl}\partial_lu_i)\big]
=\partial_{m}\big[\frac{n(n-2)}{4}\sqrt{\det g_{kl}} u_i^{\frac{n+2}{n-2}}\big].
\end{equation*}
Then,
\begin{align*}\partial_k(\big[\sqrt{\det g_{kl}}g^{kl}\partial_l\big(\partial_m u_i\big)\big]
=&-\partial_k(\big[\partial_{m}(\sqrt{\det g_{kl}}g^{kl})\partial_lu_i\big]\\
&\qquad+\partial_{m}\big[\frac{n(n-2)}{4}\sqrt{\det g_{kl}} u_i^{\frac{n+2}{n-2}}\big].
\end{align*}
By interior $H^{k}$-estimates
and the Sobolev embedding, we have
\begin{equation*}
\|\nabla_{g} u_i \|_{L^{\infty}(\Omega'')}\leq C(\Omega',\Omega'')u_i^{\frac{n}{n-2}}(p).
\end{equation*}
For each $i$, we choose $p\in \Omega''$ such that $u_i(p)=\inf_{\Omega''}u_i$.
Hence,
\begin{equation}\label{in l infinity esti}
\|\nabla_{g} v_i \|_{L^{\infty}(\Omega'')}\leq C(\Omega',\Omega'').
\end{equation}
 In the following, we consider two cases.

{\it Case 2.1.} Take an arbitrary $x_0 \in \Gamma$ and we will prove $x_0\in \mathcal{B}_{\Gamma,\{\Omega_i\}}$.
For any sufficiently small $\epsilon>0$, we choose normal coordinates in a small neighborhood of
$x_0$ such that $x_0=0$
and the line segment $\{te_n;\, t\in [0,\epsilon]\}$ on the $x_n$-axis is a geodesic connecting $x_0$
and $\epsilon e_n\in \Omega_i$ for $i$ large.
Take $t^*_i\in (0,\epsilon)$ such that $t^*_ie_n\in\partial\Omega_i$ and
$te_n \in \Omega_i$ for any $t\in (t_i^*, \epsilon]$.
By \eqref{in l infinity esti},
\begin{equation}\partial_nv_i(\epsilon e_n)\leq C,\end{equation}
for some constant $C$ independent of $i$. This estimate plays the same role as \eqref{eq-compare-vi}.
Then, we take ${t}_i \in [t^*_i,\epsilon]$ such that, for any $t\in [t^*_i,\epsilon]$,
$$\partial_nv_i(te_n)\leq \partial_nv_i({t}_ie_n).$$
Hence,
\begin{align}\label{eq-v_i-first-derivative-infty}
\partial_nv_i({t}_ie_n)>\frac{v_i(\epsilon e_n)-0}{\epsilon-t^*_i}>
\frac1\epsilon v_i(\epsilon e_n)\rightarrow\infty\,\,\text{ as } i\rightarrow\infty.\end{align}
Thus, ${t}_i \in (t^*_i,\epsilon)$, and hence
\begin{align}\label{eq-v_i-second-deri-0}\partial_{nn}v_i({t}_ie_n)=0.\end{align}
We also have
\begin{equation}\label{eq-v_icontrol12}
|v_i({t}_ie_n)|\leq
\epsilon\partial_nv_i({t}_ie_n).\end{equation}
 In view of \eqref{eq-v_i-first-derivative-infty}, \eqref{eq-v_i-second-deri-0}, and \eqref{eq-v_icontrol12},
we can apply Lemma \ref{lemma-blowup-set2} to conclude
$x_0\in \mathcal{B}_{\Gamma,\{\Omega_i\}}$.

\smallskip
 {\it Case 2.2.} Take an arbitrary $x_0 \in M\setminus  \Gamma$ and
we will prove $x_0\in \mathcal{B}_{\Gamma,\{\Omega_i\}}$ if and only if $x_0\in\mathcal{S}_{g}$.
Consider $B_{4r}(x_0)\subset\subset M\setminus\Gamma$ for some $r>0$.
Then for $i$ sufficiently large, $B_{4r}(x_0)\subset\subset\Omega_i$.
 We claim
\begin{equation}\label{eq-zero-scalar}
|v_i\nabla^2_{g} v_i|\leq  C,\quad |\nabla_{g} v_i|^2\leq C  m_i^{\frac{4}{n-2}}
\quad\text{in }B_{r}(x_0).\end{equation}
Assuming the validity of \eqref{eq-zero-scalar},
by \eqref{Ricci cur  in v Manifold} and the fact that $ \mathcal{B}_{\Gamma,\{\Omega_i\}}$ is a closed set,
 for the chosen $x_0\in M\setminus\Gamma$, we conclude that
$x_0\in \mathcal{B}_{\Gamma,\{\Omega_i\}}$ if and only if there exists a sequence
$\{x_m\}$ in $M$ such that $x_m\to x_0$ and the Ricci curvature of $g$ does not vanish at each $x_m$.

We now prove \eqref{eq-zero-scalar}.
For each $i$, we choose $p_i\in \Omega_i$ such that $u_i(p_i)=\inf_{\Omega_i}u_i$ and write
$m_i=u_i(p_i)$. Then, $m_i\rightarrow0$ as $i\rightarrow\infty$.
Set $u_i=m_iw_i$. Hence, $w_i(p_i)=1$, $w_i\geq1$ in $\Omega_i$, and
\begin{equation*} \Delta_{g} w_i  =
\frac14n(n-2){m}_i^{\frac{4}{n-2}}w_i^{\frac{n+2}{n-2}}
\quad\text{in }{\Omega}_i.\end{equation*}
Set
$$A_i=\{x\in \Omega_i|\, w_i(x)<2\}.$$
We note that  $A_i\bigcap
B_{r}(x_0 ) \neq \emptyset$, for $i$ sufficiently large.
 We prove this by a contradiction argument.
If $A_i\bigcap B_{r}(x_0 )=\emptyset$, then $w_i=2$ on $\partial A_i$, and hence, by the maximum principle,
\begin{equation*}w_i(x) \geq 2-C_12^{\frac{n+2}{n-2}}m_i^{\frac{4}{n-2}}\quad\text{in }A_i,\end{equation*}
where $C_1$ is a positive constant depending only on $n$, $M$, and $r$,
independent of $i$.
We point out that a barrier function independent of $i$ can be constructed since
a common ball $B_{r}(x_0 )$ is in the complement of $A_i$ by the contradiction assumption.
Thus,
$$1 \geq 2-C_12^{\frac{n+2}{n-2}}m_i^{\frac{4}{n-2}}\quad\text{at }p_i.$$
This leads to a contradiction, since $m_i\rightarrow0$ as $i\rightarrow\infty$.
 Therefore, $w_i<2$ somewhere in $B_{r}(x_0 )$.

Applying the  Harnack inequality to the equation for $u_i$, we have $w_i\leq C_2$ in $B_{3r}(x_0)$.
Set
$$B_i=\{x\in \Omega_i|\, w_i(x)<1+2C_12^{\frac{n+2}{n-2}}m_i^{\frac{4}{n-2}}\}.$$
Then, $B_i\bigcap B_{r}(x_0 ) \neq \emptyset$ for $i$ sufficiently large.
If this is not true, by the maximum principle, we have
 \begin{equation*}w_i(x) \geq 1+2C_12^{\frac{n+2}{n-2}}m_i^{\frac{4}{n-2}}-C_12^{\frac{n+2}{n-2}}m_i^{\frac{4}{n-2}}\quad\text{in }B_i.\end{equation*}
Thus,
$$1 \geq 1+C_12^{\frac{n+2}{n-2}}m_i^{\frac{4}{n-2}}\quad\text{at }p_i,$$
which leads to a contradiction. Note that $w_i-1$ is a nonnegative solution of
 \begin{equation*} \Delta_{g}\big( w_i-1\big) =
\frac14n(n-2){m}_i^{\frac{4}{n-2}}w_i^{\frac{n+2}{n-2}}
\quad\text{in }{\Omega}_i.\end{equation*}
By the Harnack inequality, we have
$$\max\big( w_i-1\big) \leq C_3 \Big(2C_12^{\frac{n+2}{n-2}}m_i^{\frac{4}{n-2}}+ C_2^{\frac{n+2}{n-2}}m_i^{\frac{4}{n-2}}  \Big)
\quad\text{in } B_{2r}(x_0),$$
where $C_3$ is a positive constant depending only on $n$, $M$ and $r$.
Hence,
\begin{equation*}w_i-1 \leq C m_i^{\frac{4}{n-2}}\quad\text{in } B_{2r}(x_0),\end{equation*}
where $C$ is a positive constant depending only on $n$, $M$ and $r$. By standard interior estimates, we obtain
\begin{equation*}|\nabla^2_{g} w_i|+|\nabla_{g} w_i|\leq C  m_i^{\frac{4}{n-2}}\quad\text{in } B_{r}(x_0).\end{equation*}
 This implies \eqref{eq-zero-scalar}.
\end{proof}

In Theorem \ref{nonpositive yam inv}, we can characterize blow-up sets precisely when the Yamabe invariant is nonpositive.
Next, we turn our attention to manifolds with a positive Yamabe invariant.
This is a much more complicated case.
We will decompose the manifold $M$ in three classes of points,
accumulation points of the limit set $\Gamma$, isolated points of $\Gamma$,
and points not in $\Gamma$. We first study accumulation points.
In the next result, we prove that any accumulation points of the limit set are
in the blow-up set.

\begin{theorem}\label{accumulation points}
Let $(M, g)$, $\Gamma$, and $\{\Omega_i\}$ be as in Assumption \ref{assumption-basic},
$\lambda(M,[g])$ be the Yamabe invariant of $(M, g)$,
and $\Gamma=\Gamma_{1}\bigcup\Gamma_{2}$
for $\Gamma_1$ and $\Gamma_2$ as in \eqref{eq-def-Gamma1-2}.
If $\lambda(M, [g])>0$, then any accumulation points of $\Gamma$ belong to $\mathcal{B}_{\Gamma,\{\Omega_i\}}$.
\end{theorem}

\begin{proof} By the solution of the Yamabe problem \cite{LeeParker1987}, we can assume the scalar curvature $S_g$
of $M$ is the constant $\lambda(M,[g])$.
Since $M$ is compact, we take $\Lambda>0$ such that
$$|R_{ij}|\leq \Lambda g_{ij}.$$
Let  $u_i$ be the solution
of \eqref{eq-MEq-i}-\eqref{eq-MBoundary-i} in $\Omega_i$ and set $v_i=u_i^{-\frac{2}{n-2}}$.
Then,
$g_{i}=u_i^{\frac{4}{n-2}}g=v_i^{-2}g.$
By Lemma \ref{lemme-convergence1}, we have
\begin{equation*}%\label{u_i go to zero-plus}
u_i\rightarrow 0\quad\text{in }C^{m}_{\mathrm{loc}}( M \setminus\Gamma)
\text{ as }i\rightarrow\infty,\end{equation*}
and hence
\begin{equation*}
v_i\rightarrow \infty
\quad\text{uniformly in any compact subsets of $M \setminus\Gamma$
as }i\rightarrow\infty.\end{equation*}
Fix an accumulation point $x_0 \in \Gamma$. We will prove $x_0\in \mathcal{B}_{\Gamma,\{\Omega_i\}}$.

For a sufficiently small $\epsilon_0>0$, take a point $x\in \Gamma$
such that $d_{g}(x,x_0)=2\epsilon<\epsilon_0$.
We choose normal coordinates in a small neighborhood of
$x_0$ such that $x_0=0$, $x=2\epsilon e_n$, and the line segment
$\{te_n;\, t\in [0,2\epsilon]\}$ is the  shortest geodesic connecting $x_0$ and $x$.
Since the Hausdorff dimension of $\Gamma$ is less than or equal to $(n-2)/2$,
up to a rotation in $x_1, \cdots, x_{n-1}$, we can find a curve $\sigma$ given by
$$\sigma(t)=\big(\big[R^2-(t-\epsilon)^2\big]^{1/2}
-\big[R^2-\epsilon^2\big]^{1/2}, 0, \cdots, 0, t\big)\quad\text{for }t \in [0,2\epsilon],$$
where $R$ is a sufficiently large constant such that $\sigma\bigcap\Gamma\neq \sigma$.
We take $\sigma(\overline{t})\notin\Gamma$,
for some $\overline{t}\in (0,2\epsilon)$.
Then, $\sigma(\overline{t})\in\Omega_i$ for $i$ large, and
$(v_i\circ\sigma)(\overline{t})\rightarrow\infty$ as $i\rightarrow\infty$.
We point out that $R$ and $\overline{t}$ depend on $\epsilon$, but independent of $i$.

For $i$ large, take ${t}_i^*\in (0, \overline{t})$ and ${t}'_i\in (\overline{t}, 2\epsilon)$
such that $\sigma({t}^*_i), \sigma({t}'_i)\in \partial\Omega_i$ and
$\sigma(t) \in \Omega_i$  for any $t\in ({t}^*_i,{t}'_i)$.
By the polyhomogeneous expansion of $v_i$,
we have
$$|\partial_{n} v_i(\sigma({t}^*_i))|\leq C,\quad
|\partial_{n} v_i(\sigma({t}'_i))|\leq C,$$
where $C$ is a positive constant independent of $i$ and $\epsilon$.

Consider the single variable function $(v_i\circ\sigma)(t)$ on $[{t}^*_i,{t}'_i]$.
For $\epsilon$ small and $i$ large,
we take $t_i \in [{t}^*_i,{t}'_i]$ such that, for any $t\in [{t}^*_i,{t}'_i]$,
$$\partial_{t}(v_i\circ\sigma)(t)\leq \partial_{t} (v_i\circ\sigma)(t_i).$$
Then,
$$\partial_{t} (v_i\circ\sigma)(t_i)>\frac{1}{2\epsilon} (v_i\circ\sigma)(\overline{t}).$$
Thus, $t_i \in ({t}^*_i,{t}'_i)$, and hence
\begin{equation}\label{eq-second deriv}\partial_{tt} (v_i\circ\sigma)(t_i)=0.\end{equation}
We also have
\begin{equation}\label{eq-v_icontrol2}|(v_i\circ\sigma)(t_i)|\leq
2\epsilon\partial_{t} (v_i\circ\sigma)(t_i).\end{equation}
Note that
$$\partial_{t}(v_i\circ\sigma)(t_i)=
(\partial_n v_i)(\sigma(t_i))
-\frac{t_i-\epsilon}{[R^2-(t_i-\epsilon)^2]^{1/2}}(\partial_1 v_i)(\sigma(t_i)).$$
Set
$$\nu_i=\partial_n-\frac{t_i-\epsilon}{[R^2-(t_i-\epsilon)^2]^{1/2}}\partial_1.$$
By \eqref{eq-second deriv}, we have
\begin{equation}\label{eq-second deriv-v3}( \partial_{\nu_i\nu_i}v_i)(\sigma_k(t_i))
=-\Big(\frac{1}{[R^2-(t_i-\epsilon)^2]^{1/2}}-\frac{(t_i-\epsilon)^2}{[R^2-(t_i-\epsilon)^2]^{3/2}}\Big)
\partial_1 v_i(\sigma(t_i)).\end{equation}
Hence, $( \partial_{\nu_i\nu_i}v_i)(\sigma(t_i))$ is sufficiently small compared with $|\nabla v_i|(\sigma(t_i))$,
for $R$ sufficiently large and
$\epsilon$ sufficiently small.

Since $x_0\in \Gamma$ is an accumulation point of $\Gamma$,
we take a sequence of points $x_k\in\Gamma$ with $d_{g}(x_k,x_0)=2\epsilon_k\to0$.
Correspondingly, we have a sequence of curves $\sigma_k(t)$ defined for $t\in [0,2\epsilon_k]$, with
$\sigma_k(0)=x_0$ and $\sigma_k(2\epsilon_k)=x_k$, and a sequence of
$t_i^k\in (0,2\epsilon_k)$ for each $i$, as above.
Note that $\sigma_k(t_i^k)\rightarrow x_0$ as $\epsilon_k\to 0$.
 We can apply Lemma \ref{lemma-blowup-set2} to conclude
$x_0\in \mathcal{B}_{\Gamma,\{\Omega_i\}}$.
\end{proof}

The proof of Theorem \ref{gemma3 nonempty} and Theorem \ref{accumulation points} are modified from \cite{HanShen3}.

\section{Isolated points}\label{Isolated points}

Theorem \ref{accumulation points} asserts that any accumulation points of the limit set are
in the blow-up set. A natural question is whether we can characterize other points in the blow-up set.
In this section, we study blow-up phenomena near isolated points
if the underlying manifolds have a positive Yamabe invariant.
In certain cases, we need the positive mass theorem.

 Throughout this section, we denote by $\lambda(M,[g])$ the Yamabe invariant of $(M, g)$,
and by $L_g$ the conformal Laplacian of $(M,g)$, i.e.,
\begin{equation*}L_g =-\Delta_{g} +\frac{n-2}{4(n-1)}S_g.\end{equation*}
For any $x_0\in M$, $G_{x_0}\in C^{\infty}(M \setminus\
\{x_0\})$ is the Green's function for the conformal Laplacian $L_g$ with a pole at $x_0$.

We first study manifolds not conformally equivalent to the standard sphere $S^n$.

\begin{lemma}\label{lemma-not the sphere}
Let $(M, g)$, $\Gamma$, $\{\Omega_i\}$, and $\{g_i\}$ be as in Assumption \ref{assumption-basic},
with $0<\lambda(M,[g])< \lambda(S^n,[g_{S^n}])$, $u_i$ be the solution
of \eqref{eq-MEq-i}-\eqref{eq-MBoundary-i},
and $x_0\in\Gamma$ be an isolated point of $\Gamma$.
For some constant $r>0$, assume
\begin{equation}\label{not sphere u_i difference-behavioer00}u_i=m_i{\overline{u}}_i(aG_{x_0}+b) \quad\text{in }\Omega_i\cap B_r(x_0),
\end{equation}
where $\{m_i\}$ is a sequence of positive constants with $m_i\to 0$, $\{\overline{u}_i\}$ is a sequence of functions
with $\overline{u}_i\to 1$ in $C^{m}_{\mathrm{loc}}(B_r(x_0)\setminus \{x_0\})$ as $i\to\infty$ for any $m$,
$a$ is a nonnegative constant, and $b$ is a nonnegative smooth function in $B_r(x_0)$.
If either $a$ is a positive constant or $b$ is a positive function in $B_r(x_0)$, then $x_0\in \mathcal{B}_{\Gamma,\{\Omega_i\}}$.
\end{lemma}

\begin{proof}
We consider two cases.

If $a=0$,
then $b$ is a positive function in ${B_{\delta}(x_0 )}$
and $u_i/m_i\rightarrow b$ in $C^{m}_{\mathrm{loc}}({B_{\delta}(x_0 )}\setminus\{x_0\})$
as $i\rightarrow\infty.$
Arguing similarly as in the proof of Lemma \ref{lemma-blowup-set1},
we have $x_0\in \mathcal{B}_{\Gamma,\{\Omega_i\}}.$

If  $a>0$, we can write $b(x)=A+O(|x-x_0|)$ near $x_0$
for some constant $A\geq0$.
Proceeding as in Case 3 in the proof of Theorem 4.3 in \cite{HanShen3},
we conclude $x_0 \in  \mathcal{B}_{\Gamma,\{\Omega_i\}}$.
\end{proof}

The proof of Case 3  in Theorem 4.3 in \cite{HanShen3} needs the positive mass theorem.
The sign of $b$ in \eqref{not sphere u_i difference-behavioer00}
is important and ensures a nonnegative contribution to the mass.
The positive mass theorem has been known to be true if $3\leq n\leq 7$, or $M$ is locally conformally flat, or
$M$ is spin. (See \cite{LeeParker1987}, \cite{SchoenYau1979}, \cite{SchoenYau1988}, and \cite{Witten1981}.)
These conditions are technical and can be removed according to the recent papers
\cite{Lohkamp4} and \cite{SchoenYau2017}.

In fact, the first two authors \cite{HanShen3}
proved that $\mathcal{B}_{\Gamma,\{\Omega_i\}}$ contains
an entire neighborhood of $x_0$ if $3\le n\le 5$ or $M$ is conformally flat in a neighborhood of $x_0$
%or the Weyl tensor of $g$ does not vanish at $x_0$
and that $\mathcal{B}_{\Gamma,\{\Omega_i\}}$ satisfies
$|\mathcal{B}_{\Gamma,\{\Omega_i\}}\cap B_r(x_0)|/|B_r(x_0)|\to 1$ as $r\to 0$
%$\mathcal{B}_{\Gamma,\{\Omega_i\}}$ contains a cone-shaped domain with the vertex $x_0$
in the rest of the cases.
In all cases, the blow-up set $\mathcal{B}_{\Gamma,\{\Omega_i\}}$ contains points {\it not} in $\Gamma$.
We will prove that the blow-up set $\mathcal{B}_{\Gamma,\{\Omega_i\}}$ is the {\it entire} manifold
if, in addition, $(M,g)$ is assumed to be a locally conformally flat manifold.
Refer to Theorem \ref{LCF} for details.

We now prove all isolated points in the limit set belong to the blow-up set
if the underlying manifold is not conformally equivalent to the standard sphere $S^n$.

\begin{theorem}\label{not the sphere}
Let $(M, g)$, $\Gamma$, and $\{\Omega_i\}$ be as in Assumption \ref{assumption-basic},
and $\Gamma=\Gamma_{1}\bigcup\Gamma_{2}$
for $\Gamma_1$ and $\Gamma_2$ as in \eqref{eq-def-Gamma1-2}.
If $0<\lambda(M,[g])< \lambda(S^n,[g_{S^n}])$,
then $\Gamma\subseteq \mathcal{B}_{\Gamma,\{\Omega_i\}}$.
\end{theorem}

\begin{proof} Let  $u_i$ be the solution
of \eqref{eq-MEq-i}-\eqref{eq-MBoundary-i} in $\Omega_i$ and set $v_i=u_i^{-\frac{2}{n-2}}$.
Then,
$g_{i}=u_i^{\frac{4}{n-2}}g=v_i^{-2}g,$
and
\begin{equation*}  -L_gu_{i} =
\frac14n(n-2){u}_{i}^{\frac{n+2}{n-2}}
\quad\text{in }{\Omega}_i.\end{equation*}
We assume
$$|R_{ij}|\leq \Lambda g_{ij},$$
for some positive constant $\Lambda$.
Since all accumulation points of $\Gamma$ belong to $\mathcal{B}_{\Gamma,\{\Omega_i\}}$
by Theorem \ref{accumulation points},
we only need to prove that isolated points in $\Gamma$ belong to $\mathcal{B}_{\Gamma,\{\Omega_i\}}$.

Take an isolated point $x_0\in\Gamma$. We will prove $x_0 \in  \mathcal{B}_{\Gamma,\{g_i\}}$.

{\it Case 1}: $\Gamma=\{x_0\}$.
Take a small positive constant $\delta$ such that $M\setminus \Omega_i\subset B_{\delta/2}(x_0)$ for large $i$.
For $i$ large, we denote by $m_i$
the minimum of $u_i$
in $\overline{B_{\delta}(x_0 )}\setminus B_{\delta/2}(x_0 )$.
By Lemma \ref{lemme-convergence1}, $m_i\rightarrow0$ as $i\to\infty$.
For any $\Omega'\subset\subset M\setminus \{x_0\}$ and $i$ large, by
the Harnack inequality, we have
$$u_i \leq C(\Omega') m_i
\quad\text{in } \Omega',$$
where $C(\Omega')$ is a positive constant depending only on $n$, $\Omega'$, and $M$.
We rewrite the equation for $u_i$ as
\begin{equation*}  -L_g\Big(\frac{u_i}{m_i}\Big) =
\frac14n(n-2){m}_i^{\frac{4}{n-2}}\Big(\frac{u_i}{m_i}\Big)^{\frac{n+2}{n-2}}
\quad\text{in }{\Omega}_i.\end{equation*}
By interior Schauder estimates and ${m}_i\rightarrow0$,
there exist  a subsequence $\{{ u}_{i}\}$, still denoted by $\{{ u}_{i}\}$,
and a positive function ${w} \in M\setminus \{x_0\}$ such that, for any positive integer $m$,
$$\frac{u_i}{m_i}\rightarrow w\quad\text{in }C^{m}_{\mathrm{loc}}( M\setminus \{x_0\})
\text{ as }i\rightarrow\infty,$$
and
\begin{equation*}L_g{w} =0
\quad\text{in  }  M\setminus \{x_0\}.\end{equation*}
According to Proposition 9.1 in \cite{LiZhu1999}, we conclude
\begin{equation}\label{w expression-v0a}
w= {a}G_{x_0}+{b}\quad\text{in }\, M\setminus \{x_0\},\end{equation}
where $G_{x_0}\in C^{\infty}(M \setminus\
\{x_0\})$ is the Green's function for the conformal Laplacian $L_g$ with a pole at $x_0$,
${a}$ is a nonnegative constant,
and ${b}$ is a smooth function in $M$ with
\begin{equation*}L_g {b}=0\quad\text{in } M. \end{equation*}
Since $\lambda(M,[g])>0$, we have ${b}=0$ on $M$.
By ${w}\geq 1$ in $\overline{B_{\delta}(x_0 )}\setminus B_{\delta/2}(x_0 )$, ${a}$ must be positive.
Therefore,
\begin{equation}\label{w expression-v00}
w= {a}G_{x_0}\quad\text{in }\, M\setminus \{x_0\}.\end{equation}
We write
\begin{equation}\label{not sphere u_i behavioer-v0}u_i=m_i\cdot \frac{u_i}{m_iw}\cdot w,\end{equation}
where $u_i/(m_iw)\to 1$ in $C^{m}_{\mathrm{loc}}(M\setminus \{x_0\})$ as $i\to\infty$
and $w$ is given by \eqref{w expression-v00} for some positive constant $a$.
Applying Lemma \ref{lemma-not the sphere} with $b\equiv 0$,
we conclude $x_0 \in  \mathcal{B}_{\Gamma,\{\Omega_i\}}$.

{\it Case 2}:  $\Gamma\neq \{x_0\}$.
Let $\delta$ be a small positive constant such that
$\Lambda\delta<{1}/{10}$ and $\overline{B_{\delta}(x_0 )}\bigcap\Gamma=\{x_0\}$.
Set
$$O_{i}=B_{\delta}(x_0)\setminus \overline{\Omega}_i,
\quad \widetilde{\Omega}_i=M\setminus\overline{O}_{i}.$$
Then, $\Omega_i\subseteq\widetilde{\Omega}_i$ and $\overline{O}_i\subseteq B_{\delta/2}(x_0 )$ for $i$ large,
and $\widetilde{\Omega}_i\rightarrow M\setminus \{x_0\}$
as $i\rightarrow\infty$. For $i$ large, we denote by $m_i$
the minimum of $u_i$
in $\overline{B_{\delta}(x_0 )}\setminus B_{\delta/2}(x_0 )$. Then by
Lemma \ref{lemme-convergence1}, $m_i\rightarrow0$.
For any $\Omega'\subset\subset M\setminus \Gamma$ and $i$ large, by
the Harnack inequality, we have
$$\max u_i \leq C(\Omega') m_i
\quad\text{in } \Omega',$$
where $C(\Omega')$ is a positive constant depending only on $n$, $\Omega'$, and $M$.

Let $\widetilde{u}_{i}$ be the solution of \eqref{eq-MEq}-\eqref{eq-MBoundary} for $\Omega=\widetilde{\Omega}_i$.
By the maximum principle, we have $\widetilde{u}_i\leq u_i$ in ${\Omega}_i$ for $i$ large.
For $i$ large, we denote by $\widetilde{m}_i$ the minimum of $\widetilde{u}_i$
in $\overline{B_{\delta}(x_0 )}\setminus B_{\delta/2}(x_0 )$. Then, $\widetilde{m}_i\le m_i$.
For any $\Omega'\subset\subset M\setminus \{x_0\}$ and $i$ large, by
the Harnack inequality, we have
$$\widetilde{u}_i \leq C(\Omega') {m}_i
\quad\text{in } \Omega',$$
where $C(\Omega')$ is a positive constant depending only on $n$, $\Omega'$, and $M$.
Arguing as in Case 1, by taking a subsequence if necessary,
we have, for any positive integer $m$,
\begin{equation}\label{not sphere u_i behavioer}
\frac{\widetilde{u}_{i}}{m_i}\rightarrow {a}G_{x_0}\quad\text{in }C^{m}_{\mathrm{loc}}( M\setminus \{x_0\})
\text{ as }i\rightarrow\infty,\end{equation}
where $G_{x_0}\in C^{\infty}(M \setminus\
\{x_0\})$ is the Green's function for the conformal Laplacian $L_g$ with a pole at $x_0$, and
${a}$ is a nonnegative constant.

Set
$\overline{u}_i=u_i- \widetilde{u}_i$.
Then,
\begin{equation}\label{eq-equation-i}
-L_g\overline{u}_i=f_i\quad\text{in }\Omega_i,\end{equation}
where
\begin{equation*}f_i=  \frac14n(n-2)\big[{u}_i^{\frac{n+2}{n-2}} -\widetilde{u}_i^{\frac{n+2}{n-2}}\big].\end{equation*}
Note that $\overline{u}_i\ge 0$ and $f_i\ge 0$ in $\Omega_i$.
By the polyhomogeneous expansions for $u_i$ and $\widetilde{u}_i$ as in \cite{ACF1982CMP} and \cite{Mazzeo1991},
we have
$$\overline{u}_i=d^{-\frac{n-2}{2}}O(d^{n})=O( d^{\frac{n+2}{2}})\quad\text{near }
\partial \Omega_i\cap B_{\delta}(x_0 ).$$
Since $(n+2)/2\ge 5/2$ for $n\ge 3$,
we can extend $\overline{u}_i$ to a nonnegative $C^{2}$-function in
$\overline{B_{\delta}(x_0 )}$ by defining $\overline{u}_i=0$ in $\overline{O}_i$.
Similarly, we can extend $f_i$ to a nonnegative continuous function in
$\overline{B_{\delta}(x_0 )}$ by defining $f_i=0$ in $\overline{O}_i$.
Then, $-L_g\overline{u}_i=f_i$  in $\overline{B_{\delta}(x_0 )}$ with $f_i\geq0$.
For $i$ large, we denote by $\overline{m}_i$
the minimum of $\overline{u}_i$
in $\overline{B_{\delta}(x_0 )}\setminus B_{\delta/2}(x_0 )$.
Then, $\overline{m}_i\le m_i$.
For any $r\in (0,\delta)$, by the Harnack inequality, we have
$$\overline{u}_i \leq C {m}_i
\quad\text{in } \overline{B_{\delta}(x_0)}\setminus B_{r}(x_0),$$
where $C$ is a positive constant depending only on $n$, $\delta$, $r$, and $M$.
By the maximum principle, we get
\begin{equation}\label{overline u_i contral}
\overline{u}_i\leq C {m}_i\quad\text{in }{B_{\delta}(x_0 )},\end{equation}
for some positive constant $C$ independent of $i$ large.
Next, we write \eqref{eq-equation-i} as
\begin{equation*}%\label{eq-equation-i-2}
-L_g\Big(\frac{\overline{u}_i}{m_i}\Big)=m_i^{\frac{4}{n-2}}\overline{f}_i\quad\text{in }\Omega_i,\end{equation*}
where
\begin{equation*}\overline{f}_i=
\frac14n(n-2)\Big[\Big(\frac{{u}_i}{m_i}\Big)^{\frac{n+2}{n-2}} -\Big(\frac{\widetilde{u}_i}{m_i}\Big)^{\frac{n+2}{n-2}}\Big].\end{equation*}
Note that $\overline{f}_i$ is uniformly bounded in $B_\delta(x_0)\setminus B_r(x_0)$, for any $r\in (0,\delta)$.
By taking a subsequence if necessary,
we have, for any positive integer $m$,
\begin{equation}\label{not sphere u_i difference-behavioer}
\frac{\overline{u}_{i}}{{m}_{i} }\rightarrow b
\quad\text{in }C^{m}_{\mathrm{loc}}(\overline{B_\delta(x_0)}\setminus\{x_0\})
\quad\text{as }i\rightarrow\infty,\end{equation}
where $b$ is a smooth nonnegative function in $\overline{B_{\delta}(x_0 )}\setminus\{x_0\}$ with
\begin{equation*}L_g b=0\quad\text{in } \overline{B_{\delta}(x_0 )}\setminus\{x_0\}. \end{equation*}
By \eqref{overline u_i contral}, $b$ is a bounded function in $B_{\delta}(x_0 )\setminus\{x_0\}$.
Hence, $b$ has a removable singularity at $x_0$.
In other words, $b$ can be extended to a smooth nonnegative function in $\overline{B_{\delta}(x_0 )}$ with
\begin{equation*}L_g b=0\quad\text{in } \overline{B_{\delta}(x_0 )}. \end{equation*}
By the strong maximum principle, $b$ is either identically zero or positive in $B_{\delta}(x_0 )$.

By \eqref{not sphere u_i behavioer} and \eqref{not sphere u_i difference-behavioer}, we have
\begin{equation}\label{not sphere u_i behavioer21}
\frac{u_{i}}{{m}_i }\rightarrow w
\quad\text{in }C^{m}_{\mathrm{loc}}({B_{\delta}(x_0 )}\setminus\{x_0\})
\text{ as }i\rightarrow\infty,\end{equation}
where
\begin{equation}\label{w expression-v000}w=a G_{x_0}+b.\end{equation}
We point out that $a$ is a nonnegative constant
and that $b$ is either identically zero or a positive smooth function in $B_{\delta}(x_0 )$.
Note that $w\geq1$ in $\overline{B_{\delta}(x_0 )}\setminus B_{\delta/2}(x_0 )$.
We cannot have $a=0$ and $b\equiv 0$ simultaneously.
 Similarly as \eqref{not sphere u_i behavioer-v0}, we write
\begin{equation*}
u_i={m}_i\cdot \frac{u_i}{{m}_iw}\cdot w,\end{equation*}
where $u_i/({m}_iw)\to 1$ in $C^{m}_{\mathrm{loc}}(B_\delta(x_0)\setminus \{x_0\})$ as $i\to\infty$
and $w$ is given by \eqref{w expression-v000}.
Applying Lemma \ref{lemma-not the sphere},
we conclude $x_0 \in  \mathcal{B}_{\Gamma,\{\Omega_i\}}$.
\end{proof}

In Case 2 in the proof above, assume that the ratio $\widetilde{m}_i/m_i$ converges, up to a subsequence,
to some number $\lambda$. Then, $\lambda\in [0,1]$.
Moreover, $a>0$ occurs if $\lambda\in (0,1]$,
and $a=0$ occurs if $\lambda=0$.

We note that Theorem \ref{them-rigidity} follows easily from Theorem \ref{nonpositive yam inv} and
Theorem \ref{not the sphere}.

For manifolds conformally
equivalent to the standard sphere, we have the following result.

\begin{lemma}\label{lemma-sphere}
Let $(M, g)=(S^n, g_{S^n})$, $\Gamma$, $\{\Omega_i\}$, and $\{g_i\}$ be as in Assumption \ref{assumption-basic},
$u_i$ be the solution
of \eqref{eq-MEq-i}-\eqref{eq-MBoundary-i},
and $x_0\in\Gamma$ be an isolated point of $\Gamma$.
For some constant $r>0$, assume
\begin{equation}\label{not sphere u_i difference-behavioer0}u_i=m_i{\overline{u}}_i(aG_{x_0}+b) \quad\text{in }\Omega_i\cap B_r(x_0),
\end{equation}
where $\{m_i\}$ is a sequence of positive constants with $m_i\to 0$, $\{\overline{u}_i\}$ is a sequence of functions
with $\overline{u}_i\to 1$ in $C^{m}_{\mathrm{loc}}(B_r(x_0)\setminus \{x_0\})$ as $i\to\infty$ for any $m$,
$a$ is a nonnegative constant, and $b$ is a positive smooth function in $B_r(x_0)$.
Then, $x_0\in \mathcal{B}_{\Gamma,\{\Omega_i\}}$.
\end{lemma}

The proof is similar as that of Lemma \ref{lemma-not the sphere}.
We emphasize that $b$ is assumed to be a positive function on $B_r(x_0)$.

\begin{theorem}\label{sphere}
 Let $(M, g)=(S^n, g_{S^n})$, $\Gamma$, $\{\Omega_i\}$, and $\{g_i\}$ be as in Assumption \ref{assumption-basic},
and $\Gamma=\Gamma_{1}\bigcup\Gamma_{2}$
for $\Gamma_1$ and $\Gamma_2$ as in \eqref{eq-def-Gamma1-2}.
If $\Gamma$  contains at least two points,
then either $\Gamma\subset \mathcal{B}_{\Gamma,\{\Omega_i\}}$ or
$\Gamma\setminus\{p\}\subset \mathcal{B}_{\Gamma,\{\Omega_i\}}$ for some isolated point $p\in\Gamma$.
\end{theorem}

\begin{proof}
Let  $u_i$ be the solution
of \eqref{eq-MEq-i}-\eqref{eq-MBoundary-i} in $\Omega_i$ and set $v_i=u_i^{-\frac{2}{n-2}}$.
Then,
$g_{i}=u_i^{\frac{4}{n-2}}g=v_i^{-2}g.$
Here and hereafter, we simply denote by $g$ the round metric on $S^n$.

Since the accumulation points of $\Gamma$ belong to $\mathcal{B}_{\Gamma,\{\Omega_i\}}$
by Theorem \ref{accumulation points},
we only need to study isolated points.
Assume $\Gamma$ contains at least two isolated points $x_1$ and $x_2$. We will prove at least one of them
belongs to $\mathcal{B}_{\Gamma,\{\Omega_i\}}$.

Take a sufficiently small positive constant $\delta$ such that
$n\delta<{1}/{10}$, $\overline{B_{\delta}(x_j )}\bigcap\Gamma=\{x_j\}$ for $j=1,2$,
and $\overline{B_{100\delta}(x_1)}$ and $\overline{B_{100\delta}(x_2)}$ are disjoint.
Furthermore, we take $\delta$ small such that
\begin{equation}\label{delta set}G_{x_{j}}(y)>100G_{x_{j}}(z)
\,\,\text{for any } y \in \overline{B_{\delta}(x_{j})} \setminus\{x_{j}\}\text{ and any }
z\in S^n \setminus\overline{B_{100\delta}(x_{j})},\end{equation}
where $G_{x_{j}}\in C^{\infty}(S^n\setminus\{x_j\})$ is the Green's function for the conformal Laplacian $L_{g}$
with a pole at $x_j$, for $j=1,2$.

For each $j=1,2$, since $\Omega_i$ is increasing, we have
$\overline{B_{\delta}(x_j )}\setminus B_{\delta/2}(x_j )\subset \Omega_i$ for $i$ large.
For such $i$, we denote by $m_{ij}$ the minimum of $u_i$ in $\overline{B_{\delta}(x_j )}\setminus B_{\delta/2}(x_j )$.
By Lemma \ref{lemme-convergence1}, $m_{ij}\rightarrow0$  as $i\to\infty$, for $j=1,2$.
Without loss of generality, we assume, for infinitely many $i$,
\begin{equation}\label{minimum order}m_{i1}\geq m_{i2}.\end{equation}
Taking a subsequence if necessary, we assume \eqref{minimum order} holds for all $i$ large.
We now  prove $x_2\in \mathcal{B}_{\Gamma,\{\Omega_i\}}.$
We consider two cases.

{\it Case 1}: $\Gamma=\{x_1,x_2\}$.
Arguing as in Case 1 in the proof Theorem \ref{not the sphere},
we have, for any positive integer $m$,
\begin{equation}\label{sphere u_i behavioer}\frac{u_{i}}{m_{i2} }\rightarrow w
\quad\text{in }C^{m}_{\mathrm{loc}}(S^n\setminus \{x_1,x_2\})
\text{ as }i\rightarrow\infty,\end{equation}
where
\begin{equation}\label{w expression-v000a}{w}=a_1 G_{x_1}+a_2 G_{x_2},\end{equation}
for some nonnegative constants $a_1$ and $a_2$, with $a_1+a_2>0$.

Next, we claim $a_1>0$. We consider two cases. If $a_2=0$, then $a_1>0$.
If $a_2>0$, we also have $a_1>0$. Otherwise, if $a_1=0$, then by \eqref{delta set}, we have
$w(x)\leq {1}/{100}$ in $\overline{B_{\delta}(x_1 )}\setminus B_{\delta/2}(x_1 )$, since
the minimum of $w$ on  $\overline{B_{\delta}(x_2 )}\setminus B_{\delta/2}(x_2 )$ is 1.
However, by \eqref{minimum order}, we have $w\geq 1$ in $\overline{B_{\delta}(x_1 )}\setminus B_{\delta/2}(x_1 )$,
which is a contradiction.
Therefore, we always have $a_1>0$.

We write
\begin{equation*}
u_i={m}_{i2}\cdot \frac{u_i}{{m}_{i2}w}\cdot w\quad\text{in }B_\delta(x_2)\cap \Omega_i,\end{equation*}
where $u_i/({m}_{i2}w)\to 1$ in $C^{m}_{\mathrm{loc}}(B_\delta(x_2)\setminus \{x_2\})$ as $i\to\infty$
and $w$ is given by \eqref{w expression-v000a}.
Note that $a_2\ge 0$ and $a_1 G_{x_1}$ is a positive smooth function in $B_\delta(x_2)$.
Applying Lemma \ref{lemma-sphere},
we conclude $x_2 \in  \mathcal{B}_{\Gamma,\{\Omega_i\}}$.

{\it Case 2}:  $\Gamma\neq \{x_1,x_2\}$.
Set, for $j=1,2$,
$$O_{ij}=B_{\delta}(x_j)\setminus \overline{\Omega}_i,$$
and
$$\widetilde{\Omega}_i=M\setminus\big(\overline{O}_{i1}\bigcup \overline{O}_{i2}\big).$$
Then, $\Omega_i\subseteq\widetilde{\Omega}_i$
and $\overline{O}_{ij}\subseteq B_{\delta/2}(x_j)$ for $i$ large and $j=1,2$,
and $\widetilde{\Omega}_i\rightarrow M\setminus \{x_1, x_2\}$
as $i\rightarrow\infty$.

Let $\widetilde{u}_{i}$ be the solution of \eqref{eq-MEq}-\eqref{eq-MBoundary} for $\Omega=\widetilde{\Omega}_i$.
By the maximum principle, we have $\widetilde{u}_i\leq u_i$ in ${\Omega}_i$ for $i$ large.
As in Case 1, up to a subsequence, we have, for any positive integer $m$,
\begin{equation}\label{sphere u_i behavioer-two}
\frac{\widetilde{u}_{i}}{m_{i2} }\rightarrow a_1 G_{x_1}+a_2 G_{x_2}
\quad\text{in }C^{m}_{\mathrm{loc}}(S^n\setminus \{x_1, x_2\})
\text{ as }i\rightarrow\infty,\end{equation}
where $a_1, a_2$ are nonnegative constants.

Set
$\overline{u}_i=u_i- \widetilde{u}_i$.
As in Case 2 in the proof Theorem \ref{not the sphere}, up to a subsequence,
we have, for any positive integer $m$,
\begin{equation}\label{not sphere u_i difference-behavioer-two}
\frac{\overline{u}_{i}}{{m}_{i2} }\rightarrow b
\quad\text{in }C^{m}_{\mathrm{loc}}(\overline{B_\delta(x_2)}\setminus\{x_2\})
\text{ as }i\rightarrow\infty,\end{equation}
where $b$ is either identically zero or a smooth positive function in $B_{\delta}(x_2)$.

By \eqref{sphere u_i behavioer-two} and \eqref{not sphere u_i difference-behavioer-two},
we have, for any positive integer $m$,
\begin{equation}\label{sphere u_i behavioer2}
\frac{{u}_{i}}{{m}_{i2} }\rightarrow {w}
\quad\text{in }C^{m}_{\mathrm{loc}}(\overline{B_\delta(x_2)}\setminus \{x_2\})
\text{ as }i\rightarrow\infty,\end{equation}
where
\begin{equation}\label{w expression-v000b}{w}=a_1 G_{x_1}+a_2 G_{x_2}+b,\end{equation}
for nonnegative constants $a_1$ and $a_2$ and a smooth function $b$ in $B_{\delta}(x_2)$,
either identically zero or positive in $B_{\delta}(x_2)$.

 We now claim either $a_1>0$ or $b$ is a positive smooth function in $B_{\delta}(x_2)$.
We consider two cases.
For the first case, we assume  $a_1+a_2>0$.
If $a_2=0$, then $a_1>0$.
Consider $a_2>0$.
If $a_1=0$, by a similar argument as in Case 1, we have $b>0$ in $B_\delta(x_0)$.
For the second case, we assume  $a_1+a_2=0$. Then, $a_1=a_2=0$ and $b>0$ in $B_{\delta}(x_2)$.
Therefore, $a_1 G_{x_1}+b$ is a positive smooth function in $B_{\delta}(x_2)$.

We write
\begin{equation*}
u_i={m}_{i2}\cdot \frac{u_i}{{m}_{i2}w}\cdot w\quad\text{in }B_\delta(x_2)\cap \Omega_i,\end{equation*}
where $u_i/({m}_{i2}w)\to 1$ in $C^{m}_{\mathrm{loc}}(B_\delta(x_2)\setminus \{x_2\})$ as $i\to\infty$
and $w$ is given by \eqref{w expression-v000b}.
Note that $a_2\ge 0$ and $a_1 G_{x_1}+b$ is a positive smooth function in $B_\delta(x_2)$.
Applying Lemma \ref{lemma-sphere},
we conclude $x_2 \in  \mathcal{B}_{\Gamma,\{\Omega_i\}}$.
\end{proof}

Theorem \ref{not the sphere} asserts that on manifolds not conformally
equivalent to the standard sphere, all isolated points in the limit set belong to the blow-up set.
While Theorem \ref{sphere} asserts that on manifolds conformally equivalent to the standard sphere,
all but possibly one isolated point belong to the blow-up set.
Theorem \ref{sphere} is optimal. In fact, if the limit set $\Gamma$ consists of two distinct points $p,q\in S^n$,
with different sequences of exhausting domains,
the blow-up sets can be any one of $p$ and $q$, or $\{p,q\}$, or even the entire $S^n$.
If the limit set $\Gamma$ consists of one single point $p\in S^n$,
with different sequences of exhausting domains,
the blow-up sets can be empty or nonempty.
See Section \ref{sec-examples} for details.

\section{Points outside the Limit Set}\label{sec-complement-Gamma}

In this section, we continue to study manifolds with a positive Yamabe invariant.
As we remarked after Lemma \ref{lemma-not the sphere},
for manifolds with a positive Yamabe invariant but not conformally equivalent to the standard sphere,
near isolated points of the limit set $\Gamma$, the blow-up sets contain points not in $\Gamma$.
In certain cases, the blow-up sets contain an entire neighborhood of $x_0$.
This is sharply different from manifolds with a negative Yamabe invariant,
where the blow-up sets are precisely the limit sets according to Theorem \ref{nonpositive yam inv}.
In this section, we study whether the blow-up sets can be the entire manifold.
Throughout this section, we denote by $\lambda(M,[g])$ the Yamabe invariant of $(M, g)$.

We first associate blow-up sets with positive conformal harmonic functions.
For a positive function $w$ on $M\setminus\Gamma$, define
\begin{equation}\label{eq-definition-S-NRF-gamma}
\mathcal{S}(w)
=\text{clos}\{x\in M\setminus\Gamma;  Ric_{w^{\frac{4}{n-2}}g}(x)\neq0 \}.
\end{equation}
Obviously, $\mathcal{S}(w)$ is a closed set in $M$.

\begin{lemma}\label{lemma-not the sphere interier}
 Let $(M, g)$, $\Gamma$, and $\{\Omega_i\}$ be as in Assumption \ref{assumption-basic},
and $\Gamma=\Gamma_{1}\bigcup\Gamma_{2}$
for $\Gamma_1$ and $\Gamma_2$ as in \eqref{eq-def-Gamma1-2}.
If $\lambda(M,[g])>0$,
then there exists a positive solution $w$ of $L_{g}w=0$ in $ M  \setminus\Gamma$ such that
$\mathcal{S}(w)\subseteq \mathcal{B}_{\Gamma,\{\Omega_i\}}$.
 Moreover, $w$ has at least one nonremovable singular point in $M$.
\end{lemma}

\begin{proof}
Let  $u_i$ be the solution
of \eqref{eq-MEq} and \eqref{eq-MBoundary} in $\Omega_i$ and set $v_i=u_i^{-\frac{2}{n-2}}$.
For a fixed point $p\in\Omega_1$, set $m_i=u_i(p)$.
We rewrite the equation for $u_i$ as
\begin{equation*}  -L_g\Big(\frac{u_i}{m_i}\Big) =
\frac14n(n-2){m}_i^{\frac{4}{n-2}}\Big(\frac{u_i}{m_i}\Big)^{\frac{n+2}{n-2}}
\quad\text{in }{\Omega}_i,\end{equation*}
where $L_g$ is the conformal Laplacian of $(M,g)$.
By interior Schauder estimates and ${m}_i\rightarrow0$,
there exist a subsequence $\{{ u}_{i}\}$, still denoted by $\{{ u}_{i}\}$,
and a positive function ${w} \in M\setminus \Gamma$ such that, for any positive integer $m$,
$$\frac{u_i}{m_i}\rightarrow w\quad\text{in }C^{m}_{\mathrm{loc}}( M\setminus \Gamma)
\text{ as }i\rightarrow\infty,$$
and
\begin{equation}\label{eq-equation-h}L_g{w} =0
\quad\text{in  }  M\setminus\Gamma.\end{equation}
We write
\begin{equation*}u_i=m_i\cdot \frac{u_i}{m_iw}\cdot h,\end{equation*}
where $u_i/(m_iw)\to 1$ in $C^{m}_{\mathrm{loc}}(M\setminus \Gamma)$ as $i\to\infty$.
Then,
\begin{equation*}%\label{Ricci relation}
Ric_{u_i^{\frac{4}{n-2}}g}=m_{i}^{-\frac{4}{n-2}}\Big(Ric_{w^{\frac{4}{n-2}}g}+e_i\Big),\end{equation*}
and $e_i\rightarrow0$ as $i\rightarrow\infty$.
Therefore, $\mathcal{S}(w)\subseteq \mathcal{B}_{\Gamma,\{\Omega_i\}}$.
Last, $w$ has at least one nonremovable singular point in $M$ due to $\lambda(M,[g])>0$.
\end{proof}

We point out that the set $\mathcal{S}(w)$ is empty if the conformal metric $w^{\frac{4}{n-2}}g$ is Ricci flat.
For example, if $\Gamma=\{x_0\}$, according to Proposition 9.1 in \cite{LiZhu1999}, we conclude
$w=aG_{x_0}$ for some positive constant $a$, where $G_{x_0}\in C^{\infty}(M \setminus
\{x_0\})$ is the Green's function for the conformal Laplacian $L_g$ with the pole at $x_0$.
In this case, $\mathcal{S}(w)= \emptyset$ if $(M,[g])=(S^n,[g_{S^n}])$.
However, we are more interested in the case that $\mathcal{S}(w)$ is not empty,
and, in particular, that $\mathcal{S}(w)$ is the entire manifold $M$.
To this end, we need to study when the conformal metric $w^{\frac{4}{n-2}}g$ is not Ricci flat.

\smallskip

We first characterize Ricci flat conformal metrics in the Euclidean space.

\begin{lemma}\label{lemme-convergence2z}
Let $\Omega$ be a connected domain in $\mathbb R^n$ and $p$ be a point in $\Omega$. Suppose $u$ is a positive harmonic
function in $\Omega$. Then, the Ricci curvature of $g=u^{\frac{4}{n-2}}g_E$ vanishes in a neighborhood of $p$
if and only if $u=c|x-x_0|^{2-n}$ or $u=c$
in $\Omega$, for some positive constant $c$ and some point $x_0$ in $\mathbb R^n\setminus\Omega$.
\end{lemma}

\begin{proof}
Set $v=u^{-\frac{2}{n-2}}$ and consider the conformal metric $g=v^{-2}g_E$.
By \eqref{Ricci cur  in v Manifold0} and $\Delta u=0$, the Ricci curvature of $g$ satisfies
\begin{equation}\label{Ricci cur  in v}
R_{ii}=(n-2)vv_{ii}-\frac{n-2}{2}|\nabla v|^{2}.\end{equation}
 It is clear that the if part holds.  We now prove the only if part.

Assume the Ricci curvature of $g$ vanishes in a neighborhood of $p$. We have $v_{11}=\cdots=v_{nn}$ near $p$.
By a rotation, we have $v_{ij}=0$ at $p$ when $i\neq j$.
By the analyticity, we have $$v=a_0+\sum_{i=1}^{n}a_1^ix_i+a_2r^2+a_3r^3g_3(\theta)+...+a_kr^kg_k(\theta)+...\quad\text{near }p,$$
where $r=|x-p|$ and $\theta\in S^{n-1}$.
We will prove $g_k=0$ for $k\geq 3$.

Assume $g_k\neq0$, for some $k\geq3$, and $g_i=0$ for $3\leq i<k$. If $g_k$ has a positive maximum, without loss
of generality, assume $e_1$ is a maximum point. Then,
$$\partial_{11}v(p+re_1)=2a_2+k(k-1)r^{k-2}g(e_1)+O(r^{k-1}).$$ However,
$$\partial_{22}v(p+re_1)\leq 2a_2+kr^{k-2}g(e_1)+O(r^{k-1}),$$
which contradicts $v_{11}=v_{22}$ when $r$ is small.
Similarly, we can get a contradiction when $g_k$ has a negative minimum.
Hence,
$$v=a_2|x-x_0|^2+c_0\quad\text{near }p,$$
where $a_2$ and $c_0$ are two constants and $x_0$ is some point in $\mathbb R^n$.
Substituting into \eqref{Ricci cur  in v}, we have either $c_0=0$ or $a_2=0$. Therefore,
$u=c|x-x_0|^{2-n}$ or $u=c$
near $p$, where $c$ is some positive constant and $x_0$ is some point in $\mathbb R^n$. By the unique continuation property of $u$, we have $u=c|x-x_0|^{2-n}$ or $u=c$
in $\Omega$. In the first case, we obviously have $x_0\notin \Omega$.
\end{proof}

By the conformal invariance property \eqref{confotmal laplacian}, we have the following result.

\begin{corollary}\label{cor-limit case}
Let $(M, g)=(S^n, g_{S^n})$, $\Omega$ be a connected domain in $S^n$, and $p$ be a point in $\Omega$.
Suppose $u$ is a positive solution
of $L_{g_{S^n}} u=0$ in $\Omega.$
Then, the Ricci curvature of $g=u^{\frac{4}{n-2}}g_{S^n}$ vanishes in a neighborhood of $p$ if and only if $u=aG_{x_0}$
in $\Omega$, for some positive constant $a$ and  some point $x_0$ in $S^n\setminus\Omega$.
\end{corollary}

If $(M, g)=(S^n, g_{S^n})$, we have $\mathcal{B}_{\Gamma,\{\Omega_i\}}=S^n$
under appropriate conditions. For example, for the function $w$ as in Lemma \ref{lemma-not the sphere interier},
if $w$ is not in the form $aG_{x_0}$
for any positive constant $a$ and any point $x_0$ in $\Gamma$,
then Corollary \ref{cor-limit case} asserts that the Ricci curvature of $w^{\frac{4}{n-2}}g_{S^n}$ does not vanish identically
near any point. As a consequence, the set $\mathcal{S}(w)$ defined in \eqref{eq-definition-S-NRF-gamma}
is the entire manifold $S^n$.

We now prove the main result in this section.

\begin{theorem}\label{LCF}
Let $(M, g)$, $\Gamma$, and $\{\Omega_i\}$ be as in Assumption \ref{assumption-basic},
and $\Gamma=\Gamma_{1}\bigcup\Gamma_{2}$
for $\Gamma_1$ and $\Gamma_2$ as in \eqref{eq-def-Gamma1-2}.
Assume, in addition, that $(M, g)$ is a locally conformally flat manifold with $0<\lambda(M,[g])< \lambda(S^n,[g_{S^n}])$.
Then, $\mathcal{B}_{\Gamma,\{\Omega_i\}}=M$.
\end{theorem}

\begin{proof}
Let $w$ be as in Lemma \ref{lemma-not the sphere interier}. Then, $w$ has at least one nonremovable singular point in $M$ due to the
fact $\lambda(M,[g])>0$. We will prove $\mathcal{S}(w)=M$.

Let $\widetilde{M}$ be the universal covering of $M$ and $\pi:\widetilde{M}\rightarrow M$ be a covering map.
By the Liouville result as Corollary 4 in \cite{ChodoshLi2020} and Theorem 1.3 in \cite{LesourdUngerYau2021},
$\widetilde{M}$ is conformally equivalent to
 a domain in $S^n$ with boundary of zero Newtonian capacity.
 Since $M$ is compact and $\lambda(M,[g])< \lambda(S^n,[g_{S^n}])$, $\widetilde{M}$ is a nontrivial covering of $M$.
Hence, $\pi^*w=w\circ \pi$
has more than one nonremovable singular point in $\widetilde{M}$. Let $\Phi:\widetilde{M}\rightarrow\Phi(\widetilde{M})\subseteqq S^n$ be a developing map and $\widetilde{g}=\Phi^*(\rho^{\frac{4}{n-2}}g_{S^n})$ for some positive function $\rho$ on $\Phi(\widetilde{M})$, where $\widetilde{g}=\pi^*g$.
 By Corollary \ref{cor-limit case}, the Ricci curvature of $(\pi^*w\circ\Phi^{-1})^{\frac{4}{n-2}}\rho^{\frac{4}{n-2}}g_{S^n}$ does not vanish near any point on $\Phi(\widetilde{M})$ and hence the Ricci curvature of $(w\circ\pi)^{\frac{4}{n-2}}\widetilde{g}$ does not vanish near any point on $\widetilde{M}$.
Therefore, $\mathcal{S}(w)=M$.
\end{proof}

We note that Theorem \ref{LCF-rigidity} follows easily from Theorem \ref{LCF}.

\section{Examples}\label{sec-examples}

In this section, we present several examples of blow-up sets on $S^{n}$.
One of our main interests is the case that the limit set contains isolated points and
we investigate whether blow-up occurs at these isolated points.
 When the limit set consists of two distinct points $p_1, p_2\in S^n$, we will construct
different sequences of increasing smooth domains $\Omega_i$ in $S^n\setminus\{p_1, p_2\}$,
which converges to $S^n \setminus \{p_1, p_2\}$, such that the associated blow-up set
$\mathcal B_{\{p_1, p_2\},\{\Omega_i\} }=\{p_2\}$, or $\{p_1, p_2\}$, or the entire $S^n$.
When the limit set consists of one point $p\in S^n$, we will construct
different sequences of increasing smooth domains $\Omega_i$ in $S^n\setminus\{p\}$,
which converges to $S^n \setminus \{p\}$, such that the associated blow-up set
$\mathcal B_{\{p\},\{\Omega_i\} }$ is empty,  consists of the single point $p$, or contains at least one point
 other than $p$.
These examples demonstrate the complexity of blow-up sets on $S^n$.

For simplicity, we first construct examples in $\mathbb R^n$.
If $(M,g)=(\mathbb R^n, g_E)$,
then \eqref{eq-MEq} and \eqref{eq-MBoundary} reduce to
\begin{align}
\label{eq-MainEq} \Delta u  &= \frac14n(n-2) u^{\frac{n+2}{n-2}} \quad\text{in }\,\Omega,\\
\label{eq-MainBoundary}u&=\infty\quad\text{on }\partial \Omega.
\end{align}

\begin{example}\label{exa-blow up one point}
Let $(M,g)=(\mathbb R^n, g_E)$ and $p_1, p_2\in\mathbb R^n$.
We will construct a sequence of increasing smooth domains $\{\Omega_i\}$
in $\mathbb R^n\setminus\{p_1, p_2\}$,
which converges to $\mathbb R^n\setminus\{p_1, p_2\}$, such that
$\mathcal B_{\{p_1, p_2\},\{\Omega_i\} }=\{p_2\}.$

To this end, let $\{r_i\}$ be a positive decreasing sequence with $r_i\to 0$.
Set $\hat{r}_i=r_{i}^k$ for some $k>0$ sufficiently large, and
$$\Omega_i=\mathbb R^n\setminus(\overline{ B_{r_i}(p_1)}\bigcup \overline{B_{\hat{r}_i}(p_2)}).$$
Then, $\Omega_i$ is a sequence of increasing smooth domains, which converges to $\mathbb R^n \setminus\{p_1,p_2\}$.
Let $u_i$ be the solution of \eqref{eq-MainEq}-\eqref{eq-MainBoundary}
for $\Omega=\Omega_{i}$
and set $g_i=u_{i}^{\frac{4}{n-2}}|dx|^2$.
Moreover, set
$$\Omega_{1i}=\mathbb R^n\setminus \overline{B_{r_i}(p_1}),
\quad\Omega_{2i}=\mathbb R^n\setminus\overline{ B_{\hat{r}_i}(p_2)}.$$
Let $u_{1i}$ and $u_{2i}$ be the solution of \eqref{eq-MainEq}-\eqref{eq-MainBoundary}
for $\Omega=\Omega_{1i}$ and $\Omega_{2i}$, respectively.
Then,
\begin{equation}\label{eq-SolutionInside1}
u_{1i}(x)=\Big(\frac{2r_i}{|x-p_1|^{2}-r_i^{2}}\Big)^{\frac{n-2}{2}},\end{equation}
and
\begin{equation}\label{eq-SolutionInside2}
u_{2i}(x)=\Big(\frac{2\hat{r}_i}{|x-p_2|^{2}-\hat{r}_i^{2}}\Big)^{\frac{n-2}{2}}.\end{equation}
By the maximum principle and Lemma 2.2 in \cite{hanshen2}, we have
\begin{equation*}
u_{1i}\leq u_i\leq u_{1i} +u_{2i}\quad\text{in }\Omega_i.\end{equation*}
Set $w_i=u_i-u_{1i}$. Then,
\begin{equation*}
0\leq w_{i}\leq u_{2i}\quad\text{in }\Omega_i,\end{equation*}
and
\begin{equation}\label{eq w}
\Delta w_{i}=\frac14n(n-2) u_{1i}^{\frac{n+2}{n-2}}\big[\big(1+u_{1i}^{-1}w_{i}\big)^{\frac{n+2}{n-2}}-1\big]
:=f_{i}(w_{i}).\end{equation}
Then, for any fixed $r>0$ sufficiently small
and for all sufficiently large $i$, we have
\begin{equation}\label{w bound}
w_{i}\leq u_{2i}\leq C(r)r_i^{\frac{n-2}{2}k}\quad\text{in }\,\Omega_i\setminus  \overline{B_{r}(p_2)} ,\end{equation}
for some constant $C(r)$ depending only on $r$ and $n$.

By polyhomogeneous expansions of $u_i$ near $\partial B_{r_i}(p_1)$,
there exists a constant $\delta>0$ independent of $r_i$ such that
$$|R_{pq}^{i}|\leq 2\quad\text{in } B_{(1+\delta)r_i}(p_1)\setminus \overline{B_{r_i}(p_1)},$$
where $R_{pq}^{i}$ is a Ricci curvature component corresponding to the metric $g_i$.
For any fixed $R>0$ sufficiently large,
we have
\begin{equation}\label{u_1 bound}
\Big(\frac{2r_i}{ R^2}\Big)^{\frac{n-2}{2}}
\leq u_{1i} \leq \Big(\frac{\delta }{2}r_i\Big)^{-\frac{n-2}{2}}
\quad\text{in }B_{R}(p_1)\setminus(\overline{B_{(1+{\delta}/{2})r_i}(p_1)}\bigcup \overline{B_{r}(p_2)}).\end{equation}
For any $x\in  B_{R}(p_1)\setminus(\overline{B_{(1+\delta)r_i}(p_1)}\bigcup \overline{B_{r}(p_2)})$,
by interior Schauder estimates, we have
\begin{equation}\label{w_estimate1}
\delta r_i|\nabla w_{i}|_{L^{\infty}(B_{{\delta r_i}/{8}}(x))}\leq C\big(
|w_{i}|_{L^{\infty}(B_{{\delta r_i}/{4}}(x))}+(\delta r_i)^2|f_i(w_{i})|_{L^{\infty}(B_{{\delta r_i}/{4}}(x)}\big),\end{equation}
and
\begin{align}\label{w_ estimate2}\begin{split}
(\delta r_i)^2|\nabla^2 w_{i}|_{L^{\infty}(B_{{\delta r_i}/{16}}(x))}&\leq
C\big(|w_{i}|_{L^{\infty}(B_{{\delta r_i}/{8}}(x))}
+(\delta r_i)^2|f_{i}(w_{i})|_{L^{\infty}(B_{{\delta r_i}/{8}}(x))}\\
&\qquad+(\delta r_i)^3|\nabla f_{i}(w_{i})|_{L^{\infty}(B_{{\delta r_i}/{8}}(x))}\big).\end{split}\end{align}
Choose $k$ large such that $(n-2)(k-1)/2>n+3$.
Then, for any $x\in  B_{R}(p_1)\setminus(\overline{ B_{(1+\delta)r_i}(p_1)}\bigcup \overline{B_{r}(p_2)})$,
\begin{equation}\label{w pointwise estimate}
|\nabla w_{i}(x)|\leq C(r,R,\delta)r_i^{\frac{(n-2)k}{2}-1}, \,\,\,|\nabla^2 w_{i}(x)|\leq C(r,R,\delta)r_i^{\frac{(n-2)k}{2}-2},
\end{equation}
where $C(r,R,\delta)$ is a positive constant depending only on $n$, $\delta$, $r$, and $R$.
Substituting \eqref{eq-SolutionInside1}, \eqref{u_1 bound}, and \eqref{w pointwise estimate}
in \eqref{Ricci-cur-in-v-Manifold}, we have
$$|R_{pq}^i|\leq C(r,R,\delta)\quad\text{in }B_{R}(p_1)\setminus(\overline{B_{(1+\delta)r_i}(p_1)}\bigcup\overline{ B_{r}(p_2)}).$$
Hence, for any $x\in\mathbb R^n \setminus\{p_2\}$,
the Ricci curvature of $g_i$ near $x$ does not blow up as $i\rightarrow\infty$. \end{example}

Example \ref{exa-blow up one point} can be generalized easily to more than two points.

\begin{example}\label{exa-blow up m-1 point}
Let $(M,g)=(\mathbb R^n, g_E)$ and $p_1, \cdots, p_m\in\mathbb R^n$ be $m$ distinct points, for some $m\ge 2$.
Then, we can construct a sequence of increasing smooth domains $\{\Omega_i\}$
in $\mathbb R^n\setminus\{p_1, \cdots, p_m\}$,
which converges to $\mathbb R^n\setminus\{p_1, \cdots, p_m\}$, such that
$\mathcal B_{\{p_1, \cdots, p_m\},\{\Omega_i\} }=\{p_2, \cdots, p_m\}.$

The construction is similar as that in Example \ref{exa-blow up one point}.
We only point out that deleted balls around $p_2, \cdots, p_m$ have the same radius,
much smaller than the radius of the ball around $p_1$.
\end{example}

\begin{example}\label{exa-isolated two point blow up}
Let $(M,g)=(\mathbb R^n, g_E)$ and $p_1, p_2\in\mathbb R^n$.
We will construct a sequence of increasing smooth domains $\{\Omega_i\}$
in $\mathbb R^n\setminus\{p_1, p_2\}$,
which converges to $\mathbb R^n\setminus\{p_1, p_2\}$, such that
$\mathcal B_{\{p_1, p_2\},\{\Omega_i\}}=\{p_1, p_2\}.$

To this end, let
$r>0$ be a sufficiently small constant and $k>0$ be a sufficiently large constant.
Set $r_{2i-1}=r^{k^{2i-2}}$, $r_{2i}=r^{k^{2i}}$, $\hat{r}_{2i-1}=r^{k^{2i-1}}$, $\hat{r}_{2i}=r^{k^{2i-1}}$, and
$$\Omega_i= \mathbb R^n\setminus(\overline{B_{r_i}(p_1)}\bigcup \overline{B_{\hat{r}_i}(p_2)}).$$
Then, $\Omega_i$ is a sequence of increasing smooth domains which converges to $\mathbb R^n \setminus\{p_1,p_2\}$.
According to Example \ref{exa-blow up one point}, we have
$\mathcal B_{\{p_1, p_2\},\{\Omega_{2i-1}\}}=\{p_2\}$
and
$\mathcal B_{\{p_1, p_2\},\{\Omega_{2i}\}}=\{p_1\}.$
Therefore,
$\mathcal B_{\{p_1, p_2\},\{\Omega_i\}}=\{p_1, p_2\}.$
\end{example}

\begin{example}\label{exa-isolated point not blow up}
Let $(M,g)=(\mathbb R^n, g_E)$,
$\Gamma$ be a closed smooth embedded submanifolds
in $\mathbb R^n$ of dimension less or equal than $(n-2)/2$, and $p_0\notin \Gamma$.
We will construct a sequence of increasing smooth domains $\{\Omega_i\}$
in $\mathbb R^n\setminus(\Gamma\bigcup\{p_0\})$,
which converges to $\mathbb R^n\setminus(\Gamma\bigcup\{p_0\})$, such that
$\mathcal B_{\Gamma\bigcup\{p_0\},\{\Omega_i\} }=\Gamma.$

Let $\{\Gamma_{i}\}$ be a sequence of decreasing smooth domains  in $\mathbb R^n$
containing $\Gamma$, which converges to $\Gamma$,
and $w_i$ be the solution of \eqref{eq-MainEq}-\eqref{eq-MainBoundary}
for $\Omega=\mathbb R^n \setminus\overline{\Gamma}_i$.
We require $p_0\in \mathbb R^n \setminus\overline{\Gamma}_1$
and set $r_i=w_i(p_0)$. Then, $r_i\rightarrow0$.
For some large constant $k>0$ and $i$ large, set
$$\Omega_i=\mathbb R^n \setminus(\overline{\Gamma}_i\bigcup \overline{B_{r_i^{1/k}}(p_0)}).$$
By arguing as in Example \ref{exa-blow up one point}, we have
$\mathcal B_{\Gamma\bigcup\{p_0\},\{\Omega_i\} }=\Gamma.$
\end{example}

\begin{example}\label{exa-blow up r n}
Let $(M,g)=(\mathbb R^n, g_E)$ and $p_1, p_2\in\mathbb R^n$.
We will construct a sequence of increasing smooth domains $\{\Omega_i\}$
in $\mathbb R^n\setminus\{p_1, p_2\}$,
which converges to $\mathbb R^n\setminus\{p_1, p_2\}$, such that
$\mathcal B_{\{p_1, p_2\},\{\Omega_i\} }=\mathbb R^n.$

To this end, let $\{r_i\}$ be a positive decreasing sequence with $r_i\to 0$.
Set
$$\Omega_i=\mathbb R^n\setminus(\overline{ B_{r_i}(p_1)}\bigcup \overline{B_{{r}_i}(p_2)}),$$
and
$$\Omega_{1i}=\mathbb R^n\setminus \overline{B_{r_i}(p_1}),
\quad\Omega_{2i}=\mathbb R^n\setminus\overline{ B_{{r}_i}(p_2)}.$$
Then, $\Omega_i$ is a sequence of increasing smooth domains which converges to $\mathbb R^n \setminus\{p_1,p_2\}$.
Denote by $u_i$, $u_{1i}$, and $u_{2i}$ the solution of \eqref{eq-MainEq}-\eqref{eq-MainBoundary}
for $\Omega=\Omega_i$, $\Omega_{1i}$, and $\Omega_{2i}$, respectively.
By the maximum principle and Lemma 2.2 in \cite{hanshen2}, we have
\begin{equation}\label{u_i bound}
\max\{u_{1i},u_{2i}\} \leq u_i\leq u_{1i} +u_{2i}\quad\text{in }\Omega_{i}.\end{equation}
In particular, for a fixed point $p\in\mathbb R^n \setminus\{p_1,p_2\}$, we have
\begin{equation}\label{u_i bound2}
C^{-1}r_i^{\frac{n-2}{2}}\leq u_i(p)\leq Cr_i^{\frac{n-2}{2}},\end{equation}
when $i$ large, and
\begin{equation}\label{u_i bound3}
r_i^{-\frac{n-2}{2}}u_i(x)\rightarrow 0  \quad\text{as } x\rightarrow \infty
\text{ uniformly for }i,\end{equation}
where $C$ is a positive constant depending only on $n$, $|p-p_1|$ and $|p-p_2|$.

Set $m_i=u_i(p)$. Then, $m_i\to 0$ as $i\to\infty$.
Set $w_i=u_i/m_i$. Then,
\begin{equation*}\Delta w_{i} =
\frac14n(n-2)m_i^{\frac{4}{n-2}} w_{i}^{\frac{n+2}{n-2}} \quad\text{in }\,\Omega_i.\end{equation*}
By interior Schauder estimates and the Harnack inequality, there exist a subsequence of $\{w_i\}$,
still denoted by $\{w_i\}$,  and a positive function $w\in \mathbb R^n\setminus \{p_1,p_2\}$ such that, for any $m$,
$$w_{i}\rightarrow w\quad\text{in }C^{m}_{\mathrm{loc}}(  \mathbb R^n\setminus \{p_1,p_2\})
\text{ as }i\rightarrow\infty,$$
and
\begin{equation*}\Delta w=0  \quad\text{in  } \mathbb R^n\setminus \{p_1,p_2\}.\end{equation*}
By \eqref{u_i bound} and  \eqref{u_i bound3},
according to Bocher's Theorem and Liouville's Theorem,
 \begin{equation}\label{w expression-v3}
w=a_1|x-p_1|^{2-n} + a_2|x-p_2|^{2-n}\quad\text{in }\,\mathbb R^n\setminus \{p_1,p_2\},\end{equation}
where $a_1$ and $a_2$ are some positive constants.
Therefore, we can write
$$u_i=m_i\cdot \frac{w_i}{w}\cdot w,$$
where $w_i/w\rightarrow 1$ in $C^{m}_{\mathrm{loc}}( \mathbb R^n\setminus \{p_1,p_2\})$
as $i\rightarrow\infty.$
Then, by \eqref{Ricci-cur-in-v-Manifold} and a straightforward computation,
it is easy to verify $\mathcal B_{\{p_1, p_2\},\{\Omega_i\} }=\mathbb R^n$.
\end{example}

Next, we construct an example which shows that the high dimension part of the limit set is not contained in the
blow-up set.

\begin{example}\label{example-high dime}
Let $(M,g)=(\mathbb R^n, g_E)$, $x'=(x_1, \cdots, x_k)$, and
$x''=(x_{k+1}, \cdots, x_n)$, for $(n-2)/2<k\le n-2$.
Set
$$\Gamma=\{x=(x',x'')\in\mathbb R^n;\,  x'\in \mathbb R^k, x''=0\}.$$
A solution of \eqref{eq-MainEq}-\eqref{eq-MainBoundary}
for $\Omega= \mathbb R^n \setminus\Gamma$ is given by
\begin{equation}\label{slution k}  \underline{u}^k=\Big[\Big(k-\frac{n-2}{2}\Big)\frac{2}{n}\Big]^{\frac{n-2}{4}}
|x''|^{-\frac{n-2}{2}}.  \end{equation}
Let $\{r_i\}$ be a positive decreasing sequence with $r_i\to 0$.
Set
$$\Omega_i=\{x\in\mathbb R^n;\, x'\in \mathbb R^k,\, |x''| >r_i\},$$
and
$$\Omega_o=\{x\in\mathbb R^n;\, x'\in \mathbb R^k,\, |x''| >1\}.$$
Then, $\Omega_i\rightarrow\mathbb R^n \setminus\Gamma$  as $i\rightarrow\infty$.
Denote by $u_i$ and $u_{o}$ the solution of \eqref{eq-MainEq}-\eqref{eq-MainBoundary}
for $\Omega=\Omega_i$ and $\Omega_{o}$, respectively.
Set $s=|x''|$. We have
$$u_o(x)=u_o(0,x'')=u(s),$$
and
\begin{equation}\label{u0_u_i}u_i(x)=r_i^{-\frac{n-2}{2}}u_o({x}/{r_i}).\end{equation}
By the maximum principle, we have
$$ \underline{u}^k(x)\leq u_o(x)\leq \big(\text{dist}(x,\partial\Omega_0)\big)^{-\frac{n-2}{2}}.$$
Hence, for $s>1$,
\begin{equation} \Big[\Big(k-\frac{n-2}{2}\Big)\frac{2}{n}\Big]^{\frac{n-2}{4}}
s^{-\frac{n-2}{2}}\leq u(s)\leq  (s-1)^{-\frac{n-2}{2}},  \end{equation}
Therefore,
we have, for $s>10$,
\begin{equation}\label{u lower upper bdd} C_1
s^{-\frac{n-2}{2}}\leq u_o(x)\leq  C_2s^{-\frac{n-2}{2}},  \end{equation}
where $C_1$ and $C_2$ are two positive constants.

By interior Schauder estimates, we have, for any $x\in\mathbb R^n$ with $s=|x''|>100$,
\begin{equation}\label{u0_estimate1}
s|\nabla u_o|_{L^{\infty}(B_{{s}/{8}}(x))}\leq C
\big(|u_o|_{L^{\infty}(B_{{s}/{4}}(x))}+s^2|u_o^{\frac{n+2}{n-2}}|_{L^{\infty}(B_{{s}/{4}}(x))}\big),\end{equation}
and
\begin{align}\label{u0_ estimate2}\begin{split}
s^2|\nabla^2 u_o|_{L^{\infty}(B_{{s}/{16}}(x))}&\leq
C\big(|u_o|_{L^{\infty}(B_{{s}/{8}}(x))}
+ s^2|u_o^{\frac{n+2}{n-2}}|_{L^{\infty}(B_{{s}/{8}}(x))}\\
&\quad\qquad+s^3|\nabla \big( u_o^{\frac{n+2}{n-2}}\big) |_{L^{\infty}(B_{{s}/{8}}(x))}\big).\end{split}\end{align}
Then, for any $x\in \mathbb R^n$ with $|x''|>100$,
\begin{equation}\label{u 0pointwise estimate}
|\nabla u_o(x)|\leq Cs^{-\frac{n-2}{2}-1}, \,\,\,|\nabla^2 u_o(x)|\leq Cs^{-\frac{n-2}{2}-2},
\end{equation}
where $C$ is a positive constant depending only on $n$ and $k$.
By \eqref{u lower upper bdd}, \eqref{u 0pointwise estimate}, and \eqref{Ricci-cur-in-v-Manifold},
we conclude that the Ricci curvature of $g_o=u_o^{\frac{4}{n-2}}g_{E}$
is uniformly bounded in $\{x\in \mathbb R^{n};\,|x''|> 100\}$.
We can also verify that the Ricci curvature of $g_o$ is uniformly bounded in
$\{x\in\mathbb R^n;\, 1<|x''|\leq 1+\delta \}$, for some $\delta>0$,
by the polyhomogeneous expansions of $u_o$ near $\partial \Omega_0$, and
is uniformly bounded in $\{x\in\mathbb R^n;\, 1+\delta<|x''|\leq 100 \}$ by straightforward estimates.
Hence, the Ricci curvature of $g_o=u_o^{\frac{4}{n-2}}g_{E}$ is uniformly bounded in $\Omega_o$.

By \eqref{u0_u_i}, the Ricci curvature of $g_i=u_i^{\frac{4}{n-2}}g_{E}$ is uniformly bounded
for $i$ sufficiently large. Therefore, $\mathcal B_{\Gamma,\{\Omega_i\}}=\emptyset.$
\end{example}

By the inverse map of stereographic projections, we can construct corresponding examples on $(S^n, g_{S^n})$.
In particular, let $p_1, p_2\in S^n$ be two distinct points. Then, we can construct
different sequences of increasing smooth domains $\Omega_i$ in $S^n\setminus\{p_1, p_2\}$,
which converges to $S^n \setminus \{p_1, p_2\}$ such that
the associated blow-up set
$\mathcal B_{\{p_1, p_2\},\{\Omega_i\} }=\{p_2\}$, or $\{p_1, p_2\}$, or the entire $S^n$.

\smallskip

To end this section, we consider the case that the limit set $\Gamma\subset S^n$ consists of a single point, say the north pole $p$.
We construct two examples of increasing smooth domains $\Omega_i$ in $S^n\setminus  \{p\}$,
which converges to $S^n \setminus  \{p\}$ such that the corresponding blow-up sets $\mathcal B_{ \{p\},\{\Omega_i\} }$ are not empty.
In one example,
$\mathcal B_{ \{p\},\{\Omega_i\} }= \{p\}$; while in another example, $-p\in \mathcal B_{ \{p\},\{\Omega_i\}}$.

\begin{example}\label{exa-single-point-limit-set-point}
Let $\{r_i\}$ be a positive decreasing sequence with $r_i\to 0$.
Set $\hat{r}_i=r_{i}^k$ for some $k>0$ sufficiently large, and
$$\widetilde{\Omega}_i=\mathbb R^n\setminus(\overline{ B_{r_i}(0)}\bigcup \overline{B_{\hat{r}_i}(e_1)}).$$
Let $\widetilde{g}_i$ be
the complete conformal metric in $\widetilde{\Omega}_i$
with the constant scalar curvature $-n(n-1)$.
By Example \ref{exa-blow up one point}, the Ricci curvature of $\widetilde{g}_i$ is
bounded in $\mathbb R^n\setminus B_2(0)$ and the maximum
Ricci curvature component of $\widetilde{g}_i$ diverges to infinity somewhere in $\overline{B_2(0)}$ as $i\rightarrow\infty$.
Similar to the argument in Example \ref{exa-blow up one point}, we can prove the Ricci curvature of $\widetilde{g}_i$ is uniformly bounded
in $\mathbb R^n\setminus B_2(0)$,
not only independent of $i$ but also independent of the point.

Set
$$\underline{\Omega}_i=\{x|k_ix\in\widetilde{\Omega}_i\}
=\mathbb R^n\setminus(\overline{ B_{\frac{r_i}{k_i}}(0)}\bigcup \overline{B_{\frac{\hat{r}_i}{k_i}}(\frac{e_1}{k_i})}),$$
with $k_i>0$.
Choose a sequence $k_i\rightarrow \infty$ sufficiently fast such that $\frac{2}{k_{i+1}}<\frac{r_i}{k_i}$.
Then, $\{\underline{\Omega}_i\}$ is a sequence of increasing smooth domains which converges to $\mathbb R^n\setminus \{0\}$ and
$\mathcal B_{\{0\},\{\underline{\Omega}_i\} }=\{0\}.$
Let $I$ be the inversion transform $\frac{x}{|x|^2} $ which maps $\{\underline{\Omega}_i\}$ to a bounded domain
and maps the infinity to the origin. Set $\widehat{\Omega}_i=I(\underline{\Omega}_i)$.
Then, $\widehat{\Omega}_i$ is a sequence of increasing smooth domains which converges to $\mathbb R^n$.

Let $T$ be the stereographic projection which maps the north pole $p$ to infinity and maps the south pole $-p$ to the origin.
Set $\Omega_i=T^{-1}(\widehat{\Omega}_i)$. Then, $\Omega_i$ converges to $ S^n\setminus \{p\}$,
and the blow-up set $\mathcal{B}_{\Gamma,\{\Omega_i\}}$ consists of $p$.
\end{example}

\begin{example}\label{exa-single-point-limit-set-curves}
For $n\ge 4$, set
$$\gamma=\{(0, \cdots, 0, x_n)|\, |x_n|\ge 1\}\subset\mathbb R^n.$$
Let $\widetilde{\Omega}_i$ be a sequence of increasing bounded smooth domains in $\mathbb R^n$,
star-shaped with respect to the origin, which converges to
$\mathbb R^n\setminus \gamma$. Let $\widetilde{g}_i$ be
the complete conformal metric in $\widetilde{\Omega}_i$
with the constant scalar curvature $-n(n-1)$. By Example 5.6 in \cite{HanShen3},
the maximum
Ricci curvature component of $g_i$ diverges to infinity somewhere as $i\rightarrow\infty$.
Assume the maximum of a component of Ricci curvature of $g_i$ is achieved at $x_i\in\Omega_i$.
Set
$$\widehat{\Omega}_i=\{x|x_i+k_ix\in\widetilde{\Omega}_i\},$$
with $k_i>0$.
 Since $\{\widetilde{\Omega}_i\}$ is a sequence of increasing star-shaped bounded smooth domains,
we can choose a sequence $k_i\rightarrow0$ such that $\{\widehat{\Omega}_i\}$
is a sequence of increasing bounded smooth domains  which converges to $\mathbb R^n$.
Then, the maximum Ricci curvature component of
the complete conformal metric in $\widehat{\Omega}_i$ with the constant scalar curvature $-n(n-1)$
diverges to infinity at $0$ as $i\rightarrow\infty$.

Let $T$ be the stereographic projection which maps the north pole $p$ to infinity and maps the south pole $-p$ to the origin.
Set $\Omega_i=T^{-1}(\widehat{\Omega}_i)$. Then, $\Omega_i$ converges to $ S^n\setminus \{p\}$,
 and the south pole $-p$ is in the blow-up set $\mathcal{B}_{\Gamma,\{\Omega_i\}}$. %,
\end{example}

\end{document}